\documentclass[11pt,a4paper]{amsart}

\usepackage{style}

\usepackage{times}

\def\AA{\tilde{A}_7}
\def\EE{\tilde{E}_6}
\def\XX{\tilde{D}_6\oplus\tilde{A}_1}

\tikzstyle{nodal}=[circle,draw,fill=black,inner sep=0pt, minimum width=4pt]
\tikzstyle{half-fiber}=[rectangle,draw=black,thick,inner sep=0pt, minimum width=5pt, minimum height=5pt]
\tikzset{double distance = 2pt}

\title{Enriques surfaces of zero entropy}

\author{Gebhard Martin}
\address{Mathematisches Institut \\ Universität Bonn \\ Endenicher Allee 60 \\ 53115 Bonn \\ Germany}
\email{gmartin@math.uni-bonn.de} 

\author{Giacomo Mezzedimi}
\address{Mathematisches Institut \\ Universität Bonn \\ Endenicher Allee 60 \\ 53115 Bonn \\ Germany}
\email{mezzedim@math.uni-bonn.de}

\author{Davide Cesare Veniani}
\address{Fachbereich Mathematik \\ Universität Stuttgart \\ Pfaffenwaldring 57 \\ 70569 Stuttgart \\ Germany}
\email{davide.veniani@mathematik.uni-stuttgart.de}

\date{\today}
\subjclass[2020]{14J28 (14J27 14J50 37E30)}
\keywords{Enriques surface, genus one fibration, half-fiber, entropy, automorphism group}

\begin{document}

\begin{abstract}
    We classify Enriques surfaces of zero entropy, or, equivalently, Enriques surfaces with a virtually abelian automorphism group.
\end{abstract}

\maketitle

\section{Introduction}
By a result of Nikulin \cite{Nikulin:infinite_en} and Barth--Peters \cite{BarthPeters}, the automorphism group of a very general complex Enriques surface is isomorphic to the $2$-congruence subgroup of the orthogonal group of the lattice $E_{10} \coloneqq U \oplus E_8$. This group is infinite and not virtually solvable, that is, it does not contain a solvable subgroup of finite index (see \Cref{prop:Enriques.zero.entropy}).
On the other end of the spectrum, there are families of Enriques surfaces with finite automorphism group, classified in \cite{Katsura.Kondo.Martin,Kondo,MartinFiniteAut,Nikulin}.

Therefore it is natural to ask whether there exist families of Enriques surfaces with an infinite, but less complicated (e.g., virtually abelian) automorphism group.
Barth and Peters \cite{BarthPeters} found one such family over $\IC$, and Mukai \cite{Mukai:OWR} sketched a proof that this is the only such family. 
Our main result is a complete and characteristic-free classification of Enriques surfaces with virtually abelian automorphism group.

We say that \(X\) is of \emph{type \(\AA\)}, \emph{type \(\EE\)}, or, respectively, \emph{type \(\XX\)} if \(X\) contains \((-2)\)-curves with the following dual graphs:
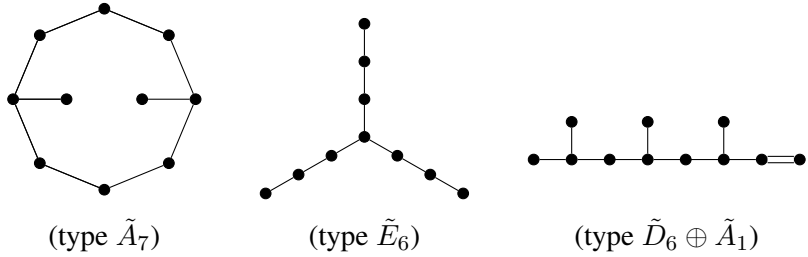
\[
    \begin{tikzpicture}[scale=0.6]
        \node (R1) at (180:2) [nodal] {};
        \node (R2) at (135:2) [nodal] {};
        \node (R3) at (90:2) [nodal] {};
        \node (R4) at (45:2) [nodal] {};
        \node (R5) at (0:2) [nodal] {};
        \node (R6) at (315:2) [nodal] {};
        \node (R7) at (270:2) [nodal] {};
        \node (R8) at (225:2) [nodal] {};
        \node (R9) at (intersection of R2--R7 and R3--R8) [nodal] {};
        \node (R10) at (intersection of R4--R7 and R3--R6) [nodal] {};
        \draw (R6)--(R7)--(R8)--(R1)--(R2)--(R3)--(R4) (R1)--(R9) (R5)--(R10) (R4)--(R5)--(R6);
        \draw (R7)--(R8)--(R1)--(R2)--(R3)--(R4) (R1)--(R9) ;
        \node at (270:3) {(type \(\tilde A_7\))};
    \end{tikzpicture}
\qquad
    \begin{tikzpicture}
        \node (R0) at (90:1.5) [nodal] {};
        \node (R1) at (90:1) [nodal] {};
        \node (R2) at (90:0.5) [nodal] {};
        \node (R3) at (0:0) [nodal] {};
        \node (R4) at (210:0.5) [nodal] {};
        \node (R5) at (210:1) [nodal] {};
        \node (R6) at (210:1.5) [nodal] {};
        \node (R7) at (330:0.5) [nodal] {};
        \node (R8) at (330:1) [nodal] {};
        \node (R9) at (330:1.5) [nodal] {};
        \draw (R0)--(R1)--(R2)--(R3)--(R4)--(R5)--(R6) (R3)--(R7)--(R8)--(R9);
        \node at (0,-1.3) {(type \(\tilde E_6\))};
    \end{tikzpicture}
\qquad
 \begin{tikzpicture}[scale=0.5]
    \node (R4) at (0,0) [nodal] {};
    \node (R5) at (1,1) [nodal] {};
    \node (R6) at (1,0) [nodal] {};
    \node (R7) at (2,0) [nodal] {};
    \node (R8) at (3,0) [nodal] {};
    \node (R9) at (4,0) [nodal] {};
    \node (R3) at (-1,0) [nodal] {};
    \node (R2) at (-2,0) [nodal] {};
    \node (R1) at (3,1) [nodal] {};
    \node (RX) at (-1,1) [nodal] {};
    \node (RXX) at (5,0) [nodal] {};
\draw (R2)--(R3)--(R6) (R5)--(R6)--(R9) (R1)--(R8) (R3)--(RX);
\draw[double] (R9)--(RXX);
    \node at (1.5,-2) {(type \(\XX\))};
\end{tikzpicture}
\]

\begin{theorem} \label{thm: main}
Let $X$ be an Enriques surface with infinite automorphism group over an algebraically closed field $k$ of characteristic $p\ge 0$. Then, the following are equivalent:
\begin{enumerate}
\item The automorphism group $\Aut(X)$ is virtually abelian.
\item The automorphism group $\Aut(X)$ is virtually solvable.
\item The Enriques surface \(X\) is of type \(\AA\), \(\EE\) or \(\XX\).
\end{enumerate}
\end{theorem}

We note that Enriques surfaces of type \(\EE\) and \(\XX\) only exist in characteristic~$2$ by \Cref{lem: halffibers}. Therefore, if $p\ne 2$, there exists a unique family of Enriques surfaces of zero entropy and infinite automorphism group, namely the family studied by Barth--Peters. This confirms Mukai's statement in \cite{Mukai:OWR}.

The classification of Enriques surfaces with virtually abelian automorphism group is closely related to the notion of \emph{entropy} of automorphisms, which we recall in \Cref{sec: entropy}. In \Cref{prop:Enriques.zero.entropy}, we prove that an Enriques surface has virtually abelian automorphism group if and only if it has zero entropy.

\begin{remark}
The classification of K3 surfaces of zero entropy has recently been completed by Yu \cite{Yu:K3.entropy} and Brandhorst--Mezzedimi \cite{Brandhorst.Mezzedimi:K3.zero_entropy}. By \cite[Remark~5.5]{Brandhorst.Mezzedimi:K3.zero_entropy}, the K3~covers of very general Enriques surfaces of type \(\AA\) have zero entropy as well. 
In particular, their automorphism group is infinite, but virtually abelian. 
This is in stark contrast to the K3 covers of Enriques surfaces with finite automorphism group, which instead have an infinite, non virtually solvable automorphism group by \cite[Corollary~5.4.(2)]{Brandhorst.Mezzedimi:K3.zero_entropy}.
\end{remark}

\begin{remark}
The automorphism group of Enriques surfaces of type \(\AA\) was computed by Barth--Peters \cite{BarthPeters}. In \cite[Theorem 4.12]{BarthPeters}, they claim that the automorphism group of such surfaces is never larger than $\IZ/4\IZ \times D_\infty$, where $D_\infty$ denotes the infinite dihedral group. This turns out to be false: indeed, we show in \Cref{prop: Aut of BP} that there exists a single surface in the family, whose automorphism group is a non-split extension of $\IZ/2\IZ$ by $\IZ/4\IZ \times D_\infty$ (cf. also \Cref{remark:BP.wrong}).
\end{remark}

This article is structured as follows. 
In \Cref{sec: preliminaries}, we recall preliminaries on genus~\(1\) fibrations on Enriques surfaces and investigate the action of the Mordell--Weil group of the Jacobian on the fibers of the fibration. We also quickly recall the notion of entropy of automorphisms. 
In \Cref{sec: examples}, we show that Enriques surfaces of type \(\AA\), \(\EE\) and \(\XX\) have zero entropy and we compute their automorphism groups and number of moduli. 
Finally, in \Cref{sec: Classification}, we prove \Cref{thm: main} using the connection with zero entropy given in \Cref{prop:Enriques.zero.entropy}.

\subsection*{Acknowledgments} 
We thank Shigeyuki Kond\=o for making us aware of Mukai's report~\cite{Mukai:OWR}. We are grateful to Igor Dolgachev and Matthias Sch\"utt for helpful comments on a first draft of this article.

\section{Preliminaries} \label{sec: preliminaries}


An \emph{Enriques surface} is a smooth and proper surface \(X\) over an algebraically closed field~\(k\) with numerically trivial canonical class \(K_X\) and \(b_2(X) = 10\). We let \(p\) be the characteristic of \(k\). For $p \ne 2$, the canonical bundle $\omega_X$ of an Enriques surface $X$ is $2$-torsion. On the other hand, recall that for $p = 2$, there are three types of Enriques surfaces, with different torsion component $\Pic^\tau_X$ of the identity of their Picard scheme: \emph{classical}, with $\Pic^\tau_X \cong \IZ/2\IZ$, \emph{ordinary}, with $\Pic^\tau_X \cong \mu_2$, and \emph{supersingular}, with $\Pic^\tau_X \cong \alpha_2$. 
Classical Enriques surfaces have a $2$-torsion canonical bundle $\omega_X$, while for ordinary and supersingular Enriques surfaces, $\omega_X \cong \mathcal{O}_X$ is trivial.

Let us briefly explain the contents of this section. In \Cref{sec: genusone}, we collect some known results about genus~\(1\) fibrations, with particular focus on Enriques surfaces. In \Cref{sec:MW}, we define the Mordell--Weil group of a genus~\(1\) fibration, and we collect several results on the action of this group on the reducible fibers.
Finally, in \Cref{sec: entropy}, we recall the definition of the algebraic entropy of automorphisms and give a characterization of Enriques surfaces of zero entropy (cf. \Cref{prop:Enriques.zero.entropy}).

\subsection{Genus~\texorpdfstring{\(1\)}{1} fibrations on Enriques surfaces} \label{sec: genusone}


For a comprehensive account on genus~\(1\) fibrations, we refer to \cite[Chapter~4]{Enriques_I}.
Let \(X,Y\) be normal varieties over a field. A \emph{genus~\(1\) fibration} is defined as a proper, surjective and flat morphism \(f\colon X \to Y\) with \(f_*\cO_X = \cO_Y\), such that the generic fiber \(X_\eta\) is a geometrically integral and regular curve of genus~\(1\). The fibration \(f\) is called \emph{elliptic} if \(X_\eta\) is smooth, otherwise it is called \emph{quasi-elliptic}. 
Following Kodaira's notation, we recall that the non-multiple singular fibers of genus~\(1\) fibrations are either \emph{additive} and denoted by $\II,\III,\IV,\IV^*,\III^*,\II^*,$ or $\I_n^*$, or \emph{multiplicative} and denoted by $\I_n$.

Now let $X$ be an Enriques surface.
A \emph{half-fiber} on~\(X\) is a non-trivial, connected, nef divisor $F$ with $F^2 = 0$ and $h^0(F) = 1$. Every Enriques surface carries a half-fiber (see, e.g., \cite[Corollary~2.3.4]{Enriques_I}). Moreover, every genus~\(1\) fibration on~\(X\) is induced by a linear system of the form $|2F|$, where $F$ is a half-fiber on~\(X\). The following result characterizes the structure of half-fibers on~\(X\):

\begin{lemma}[{\cite[Theorem~4.10.3]{Enriques_I}}] \label{lem: genus.1.fibrations} \label{lem: halffibers}
Let \(f\colon X \rightarrow \IP^1\) be a genus~\(1\) fibration on an Enriques surface~\(X\).
\begin{itemize}
    \item If \(p \neq 2\), then \(f\) is an elliptic fibration with two half-fibers, and each of them is either non-singular, or singular of multiplicative type.
    \item If \(p = 2\) and \(X\) is classical, then \(f\) is an elliptic or quasi-elliptic fibration with two half-fibers, and each of them is either an ordinary elliptic curve, or singular of additive type.
    \item If \(p = 2\) and \(X\) is ordinary, then \(f\) is an elliptic fibration with one half-fiber, which is either an ordinary elliptic curve, or singular of multiplicative type.
    \item If \(p = 2\) and \(X\) is supersingular, then \(f\) is an elliptic or quasi-elliptic fibration with one half-fiber, which is either a supersingular elliptic curve, or singular of additive type. 
\end{itemize}
\end{lemma}



\subsection{Mordell--Weil group of the Jacobian} \label{sec:MW}
In \Cref{lem: genus.1.fibrations}, we saw that every genus~\(1\) fibration $f \colon X \to \mathbb{P}^1$ on an Enriques surface $X$ has a double fiber, hence in particular no section.
The associated Jacobian fibration $J(f) \colon J(X) \to \IP^1$ is a genus~\(1\) fibration on a rational surface $J(X)$ by \cite[Proposition 4.10.1]{Enriques_I}, and $J(f)$ is (quasi-)elliptic if and only if $f$ is (quasi-)elliptic. More precisely, by \cite[Theorem 6.6]{LiuLorenziniRaynaud}, the fibers of~\(f\) and $J(f)$ have the same Kodaira types.
The natural action of $\MW(J(f))$ on the generic fiber of~\(f\) extends to a regular action on~\(X\). If \(f \colon X \rightarrow \IP^1\) is the genus~\(1\) fibration induced by the pencil \(|2F|\), we put
\[
    \MW(|2F|) \coloneqq \MW(J(f)),
\]
and we identify this group with a subgroup of \(\Aut(X)\) (see, e.g., \cite[Theorem~3.3]{DolgachevMartin}). It is well-known how this group acts on simple fibers of~\(f\):

\begin{lemma} \label{lem: MWonJAC}
Let $f \colon X \to B$ be a genus~\(1\) fibration with Jacobian $J(f) \colon J(X) \to B$. Let $b \in B$ be a point and let $X_b$ and $J(X)_b$ be the fibers of~\(f\) and $J(f)$ over $b$. Assume that $X_b$ is simple.  
Then, there exists an $\MW(J(f))$-equivariant isomorphism $J(X)_b \cong X_b$.
\end{lemma}
\begin{proof}
Since $X_b$ is simple, $f$ admits a section over an étale neighborhood $U$ of $b \in B$. As the smooth locus of~\(f\) is a torsor under the smooth locus of $J(f)$ and $X_U$ and $J(X)_U$ are the unique relatively minimal proper regular models  of the smooth part of the respective fibration, there is a $\MW(J(f))$-equivariant isomorphism between $X_U$ and $J(X)_U$. Restricting to a point of $U$ lying over $b$, we obtain the desired isomorphism. 
\end{proof}




\begin{table}[t]
    \centering
    \begin{tabular}{lll}
        \toprule
        Reducible fibers & \(\MW\) & Action on dual graph of a reducible fiber \\
        \midrule
        \(\II^*\)   & \{0\} & trivial
        \\
        \(\I_4^*\)  & \(\IZ/2\IZ\) & reflection along central vertex \\
        \(\I_9\)    & \(\IZ/3\IZ\) & rotation of order $3$ \\
        \(\III^*\), \(\III\)  & \(\IZ/2\IZ\) & transitive on simple components\\
        \(\III^*\), \(\I_2\)  & \(\IZ/2\IZ\) & transitive on simple components\\
        \(\I_8\), \(\III\) & \(\IZ/4\IZ\) & rotation of order $4$ on $\I_8$,  transitive on $\III$\\
        \(\I_8\), \(\I_2\) & \(\IZ/4\IZ\) & rotation of order $4$ on $\I_8$,  transitive on $\I_2$\\
        \(\IV^*\), \(\IV\) & \(\IZ/3\IZ\) & transitive on simple components\\
        \(\IV^*\), \(\I_3\) & \(\IZ/3\IZ\) & transitive on simple components\\
        \(\I_1^*\), \(\I_4\) & \(\IZ/4\IZ\) & transitive on simple components\\
        \(\I_0^*\), \(\I_0^*\) & \((\IZ/2\IZ)^2\) & transitive on simple components\\
        \(\I_5\), \(\I_5\) & \(\IZ/5\IZ\) & transitive on simple components \\
        \(\I_2^*\), \(\I_2\), \(\I_2\) & \((\IZ/2\IZ)^2\) & transitive on simple components \\
        \(\I_6\), \(\IV\), \(\I_2\) & \(\IZ/6\IZ\) & transitive on simple components \\
        \(\I_6\), \(\I_3\), \(\III\) & \(\IZ/6\IZ\) & transitive on simple components \\
        \(\I_6\), \(\I_3\), \(\I_2\) & \(\IZ/6\IZ\) & transitive on simple components \\
        \(\I_4\), \(\I_4\), \(\I_2\), \(\I_2\) & \(\IZ/2\IZ \times \IZ/4\IZ\) &transitive on simple components \\
        \(\I_3\), \(\I_3\), \(\I_3\), \(\I_3\) & \((\IZ/3\IZ)^2\) &transitive on simple components \\
        \bottomrule
    \end{tabular}
    \caption{Extremal elliptic fibrations on rational surfaces.}
    \label{tab:elliptic}
\end{table}

\begin{table}[t]
    \centering
    \begin{tabular}{lll}
        \toprule
        Reducible fibers & \(\MW\) & Action on dual graph of a reducible fiber \\
        \midrule
        \(\II^*\)   & \{0\}  & trivial \\
        \(\I_4^*\)  & \(\IZ/2\IZ\)  & reflection along central vertex \\
        \(\III^*\), \(\III\)   & \(\IZ/2\IZ\)  & transitive on simple components \\
        \(\I_0^*\), \(\I_0^*\) & \((\IZ/2\IZ)^2\)  & transitive on simple components \\
        \(\I_2^*\), \(\III\), \(\III\) & \((\IZ/2\IZ)^2\)  & transitive on simple components \\
        \(\I_0^*\), \(4\times \III\) & \((\IZ/2\IZ)^3\)  & transitive on simple components \\
        \(8 \times \III\) & \((\IZ/2\IZ)^4\)  & transitive on simple components \\
        \bottomrule
    \end{tabular}
    \caption{Quasi-elliptic fibrations on rational surfaces in characteristic $2$.}
    \label{tab:quasi-elliptic}
\end{table}

We say that a genus~\(1\) fibration \(f\) is \emph{extremal} if \(\MW(J(f))\) is a finite group.
Any quasi-elliptic fibration of a smooth and proper surface is extremal (see, e.g., \cite[Theorem 4.3.3]{Enriques_I}). 
It will turn out that extremal rational genus~\(1\) fibrations with $2$-elementary Mordell--Weil group play a fundamental role in the classification of Enriques surfaces of zero entropy. For the convenience of the reader, we recall in \Cref{tab:elliptic}  and \Cref{tab:quasi-elliptic} the classification of extremal elliptic and quasi-elliptic fibrations on rational surfaces (cf. \cite{Ito:char2, Lang1, Lang2, MirandaPersson, Naruki}). 
Furthermore, we know exactly how the sections of such a fibration meet the reducible fibers, and thus, using \Cref{lem: MWonJAC}, we observe the following: 

\begin{corollary} \label{cor: actiononsimplefiber}
Let $f \colon X \to \mathbb{P}^1$ be an extremal genus~\(1\) fibration of an Enriques surface $X$. If \(G\) is a simple reducible fiber of \(f\), then $\MW(J(f))$ acts on the dual graph of \(G\) as in \Cref{tab:elliptic} and \Cref{tab:quasi-elliptic}.
\end{corollary}

Describing the action of $\MW(J(f))$ on the half-fibers of~\(f\) is, in general, more delicate. If $f$ admits multiplicative half-fibers, then one can use the K3 cover to study this action. Recall that the symmetry group of the dual graph of a configuration of type~$\I_n$ with $n \geq 3$ is isomorphic to $D_{2n}$, the dihedral group of order $2n$. In analogy with the classical representation of $D_{2n} \cong \mathbb{Z}/n\mathbb{Z} \rtimes \mathbb{Z}/2\mathbb{Z}$, we call elements in the first factor \emph{rotations} and all other elements \emph{reflections}.

\begin{lemma} \label{lem: action.on.In}
Let $f \colon X \to \mathbb{P}^1$ be an elliptic fibration of an Enriques surface $X$. Let $F$ be a half-fiber of~\(f\) and assume that $F$ is of type~$\I_n$. 
Then, the following hold:
\begin{enumerate}
\item If $g \in \Aut(X) \setminus \MW(J(f))$ is an involution preserving each fiber of~\(f\), then $g$ acts as a reflection on~\(F\).
\item If $g \in \MW(J(f))$ acts as a rotation of odd order $r$ on the fiber of $J(f)$ corresponding to $F$, then it acts as a rotation of order $r$ on~\(F\).
\item If $g \in \MW(J(f))$ acts as a rotation of even order $r$ on the fiber of $J(f)$ corresponding to $F$, then it acts as a rotation of order $r/2$ on~\(F\). 
\end{enumerate}
\end{lemma}

\begin{proof}
Since $f$ admits a half-fiber of type~$\I_n$, the Enriques surface $X$ is ordinary if $p = 2$. 
Thus, the K3 cover $\pi \colon \widetilde{X} \to X$ is \'etale and a quotient by a fixed point free involution $\tau$. Since $F$ is half-fiber, $\pi^{-1}(F)$ is a (necessarily simple) fiber of an elliptic fibration $\widetilde{f}$ on $\widetilde{X}$, and since $\pi^{-1}(F) \to F$ is \'etale, $\pi^{-1}(F)$ is of type~$\I_{2n}$. The only fixed point free involution of such a configuration is a rotation of order $2$, hence $\tau$ acts as such a rotation. The preimage $\pi^{-1}(E)$ of a component $E$ of~\(F\) is the union of two components $\widetilde{E}$ and $\widetilde{E}'$ on opposite sides of $\pi^{-1}(F)$.

Now, for Claim (1), observe that $g$ lifts to an automorphism $\widetilde{g}$ of $\widetilde{X}$ that preserves the fibers of $\widetilde{f}$ and is not a translation. Since $\pi^{-1}(F)$ is a simple fiber, $\widetilde{g}$ acts as a reflection on $\pi^{-1}(F)$. Taking the quotient by $\tau$, we see that $g$ acts as a reflection on~\(F\), as claimed.

For Claims (2) and (3), observe that we can realize the Jacobian $J(\widetilde{f}) \colon J(\widetilde{X}) \to \mathbb{P}^1$ as the minimal resolution of the base change of $J(f) \colon J(X) \to \mathbb{P}^1$ along the morphism $\mathbb{P}^1 \to \mathbb{P}^1$ given by the finite part of the Stein factorization of $f \circ \pi$. We obtain a generically finite morphism $\pi' \colon J(\widetilde{X}) \to J(X)$. Let $F'$ be the fiber of $J(f)$ corresponding to $F$.
By \Cref{lem: MWonJAC} and since $\pi^{-1}(F)$ is simple, the $\MW(J(\widetilde{f}))$-action on $\pi'^{-1}(F')$ can be identified with the $\MW(J(\widetilde{f}))$-action on~$\pi^{-1}(F)$. Taking the quotient by $\tau$, we obtain Claims (2) and (3).
\end{proof}

If $f$ is quasi-elliptic, we can use the existence of the curve of cusps together with some lattice theory to understand the action on additive half-fibers.


\begin{lemma} \label{lem: action.on.additive}
Let $f\colon X \to \mathbb{P}^1$ be a quasi-elliptic fibration of an Enriques surface $X$. Let $F$ be a half-fiber of~\(f\) and let $R$ be the curve of cusps of~\(f\). Then, the following hold:
\begin{enumerate}
\item The Mordell--Weil group $\MW(J(f))$ is $2$-elementary.
\item The group $\MW(J(f))$ preserves every component of~\(F\).
\end{enumerate}
\end{lemma}

\begin{proof}
Claim (1) follows from \Cref{tab:quasi-elliptic}.



For Claim (2), we use the fact that $F$ is of type~$\II^*, \III^*, \I_4^*,\I_2^*,\I_0^*$ or $\III$. If $F$ is not of type~$\I_{2n}^*$, then there are at most two simple components in $F$, and $R$ meets only one of them, so the group $\MW(J(f))$ preserves all simple components of~\(F\) and, consequently, it preserves all components of~\(F\).

Assume instead that $F$ is of type~$\I_{2n}^*$, and denote by $C_0,\ldots,C_3$ the four simple components of~\(F\), with $C_0$ being the one meeting $R$. Every involution $\sigma \in \MW(J(f))$ preserves another simple component of~\(F\), say $C_1$, and if $n>0$ this must be the simple component near $C_0$. This implies that $\sigma$ preserves all double components of~\(F\), and $\sigma$ either preserves $C_2$ and $C_3$ as well, or it swaps them.

If $n=0$, then $\sigma$ has two fixed points on the central component. Since $p = 2$ and involutions of $\mathbb{P}^1$ in characteristic $2$ have only one fixed point, $\sigma$ fixes the central component pointwise and thus $\sigma$ preserves all components of~\(F\).


If $n=2$, then $|2F|$ has two additional reducible fibers $G_1$ and $G_2$, both of type~$\III$. We consider the invariant and coinvariant lattices of $\sigma$, which we denote by $\Num(X)^\sigma$ and $\Num(X)_\sigma \coloneqq (\Num(X)^{\sigma})^\perp$, respectively. Recall that both lattices are $2$-elementary, since $\Num(X)$ is unimodular. Assume that $\sigma$ does not preserve all components of~\(F\). Then, by considering fiber components and the curve of cusps, one easily checks that ${\rm rk}(\Num(X)^\sigma) = 9 - a$ and ${\rm rk}(\Num(X)_\sigma) = 1 + a$, where $a\in \{1,2\}$ is the number of $G_i$ whose components are permuted by $\sigma$. Moreover, we have $(-4) \oplus (-8)^a \subseteq \Num(X)_{\sigma}$, the first summand generated by $C_2 - C_3$ and the second summand generated by the difference of components of the $G_i$. This is a contradiction, since $(-4) \oplus (-8)^a$ has no $2$-elementary overlattice.

Finally, if $n=4$, then $X$ is extra-special of type~$\tilde{D}_8$, and by \cite[Remark 12.4]{Katsura.Kondo.Martin} we know that the group $\Aut(X)$ acts trivially on $\Num(X)$, so in particular it acts trivially on~\(F\).
\end{proof}

In the following, for a given genus~\(1\) fibration $f \colon X \to \mathbb{P}^1$, we let $\Aut_{\mathbb{P}^1}(X) \subseteq \Aut(X)$ be the subgroup of automorphisms of $X$ preserving $f$ and fixing the base of the fibration pointwise.

\begin{lemma} \label{lem: sign involution}
Let $f\colon X \to \mathbb{P}^1$ be a non-isotrivial elliptic fibration of an Enriques surface $X$. Then,
\[
\Aut_{\mathbb{P}^1}(X) \cong \MW(J(f)) \rtimes \mathbb{Z}/2\mathbb{Z}.
\]
Every element of $\Aut_{\mathbb{P}^1}(X)\setminus \MW(J(f))$ is an involution that acts with fixed points on a general fiber of~\(f\). If two such involutions fix a common point on a general fiber of~\(f\), they coincide. 
\end{lemma}
\begin{proof}
Let $F_{\eta}$ be the generic fiber of~\(f\). Since $X$ is the unique minimal proper regular model of $F_{\eta}$, we have $\Aut_{\mathbb{P}^1}(X) \cong \Aut(F_{\eta})$. Since $f$ is non-isotrivial and elliptic, 
the known structure of automorphisms of elliptic curves shows that $\Aut(F_{\eta}) \subseteq \MW(J(F)_\eta) \rtimes \mathbb{Z}/2\mathbb{Z}$, where the splitting is induced by identifying $\mathbb{Z}/2\mathbb{Z}$ with the stabilizer of a geometric point of $F_{\eta}$. Thus, to finish the proof, it suffices to realize an involution that is not a translation.

For this, let $F$ be a half-fiber of~\(f\) and pick a half-fiber $F_1$ on~\(X\) of some other fibration such that $F.F_1 = 1$. This is possible by \cite[Theorem 6.1.10]{Enriques_II} and because $X$ is not extra-special of type~$\tilde{E}_8$ since $f$ is elliptic (cf. \cite[Proposition~6.2.7]{Enriques_II}). 
The linear system $|2F + 2F_1|$ induces a generically finite morphism $\pi \colon X \to \mathsf{D}$ of degree $2$ by \cite[Section 3]{Enriques_I} and the pencils $|2F|$ and $|2F_1|$ are mapped to pencils of conics on $\mathsf{D}$. Since $|2F|$ is elliptic and its image on $\mathsf{D}$ is a pencil of conics by \cite[Theorem 3.3.11]{Enriques_I}, $\pi$ must be separable.
We let $g \in \Aut(X)$ be the covering involution of~$\pi$. Since the image of $|2F|$ under~$\pi$ is a pencil of conics, we deduce that $g$ preserves every member of $|2F|$ and acts with a fixed point on a general member, hence $g \in \Aut_{\mathbb{P}^1}(X)\setminus  \MW(J(f))$ and we are done. 
 \end{proof}

\subsection{Entropy} \label{sec: entropy}

Let $X$ be a smooth projective surface over an algebraically closed field $k$ of arbitrary characteristic. For an automorphism $g$ of $X$, the (\emph{algebraic}) \emph{entropy} of $g$ is defined as the logarithm of the spectral radius of the pullback $g^*$ on $\Num(X)\otimes \IC$. If the base field is $\IC$, the entropy of $g$ coincides with the topological entropy of the biholomorphism $g$ on~\(X\).

The automorphism $g$ has zero entropy if and only if all eigenvalues of the action of $g$ on $\Num(X)$ are roots of unity. This happens for instance if $g$ is periodic, i.e., if it has finite order. 
From the point of view of hyperbolic geometry, $g$ has zero entropy if and only if the isometry $g^*\in {\rm O}^+(\Num(X))$ induced by pullback is \emph{elliptic} (if $g^*$ has finite order) or $\emph{parabolic}$ (if $g^*$ has infinite order and preserves a nef isotropic vector in $\Num(X)$), cf. \cite[§4.7]{ratcliffe}. 
In the case of K3 surfaces, Cantat~\cite{Cantat:Dynamique} gives geometric descriptions of automorphisms of zero entropy.

We say that the surface $X$ has \emph{zero entropy} if all its automorphisms have zero entropy. In this context, surfaces of zero entropy naturally stand out as the surfaces with the simplest dynamics and the simplest infinite automorphism groups, as we are going to show now.

Recall the following characterization of Enriques surfaces with finite automorphism group:

\begin{proposition}[\cite{Katsura.Kondo.Martin,Kondo,MartinFiniteAut}] \label{prop:Enriques.finite.Aut}
Let \(X\) be an Enriques surface. Then, the automorphism group \(\Aut(X)\) is finite if and only if every genus~\(1\) fibration on \(X\) is extremal.
\end{proposition}

The following proposition characterizes Enriques surfaces of zero entropy with infinite automorphism group in an analogous way. 

\begin{proposition} \label{prop:Enriques.zero.entropy}
Let \(X\) be an Enriques surface with infinite automorphism group. Then, the following are equivalent:
\begin{enumerate}
\item The surface $X$ has zero entropy.
\item The automorphism group $\Aut(X)$ is virtually abelian.
\item The automorphism group $\Aut(X)$ is virtually solvable.
\item There exists exactly one non-extremal genus~\(1\) fibration on~\(X\).
\item There exists a genus~\(1\) fibration that is preserved by all of $\Aut(X)$.
\end{enumerate}
\end{proposition}

\begin{proof}
The proof relies on hyperbolic geometry. Denote by $\mathbb{H}_X$ the $9$-dimensional hyperbolic space associated to the hyperbolic lattice $\Num(X)$, and consider the natural homomorphism $\varphi \colon \Aut(X)\to \mathrm{O}(\Num(X))\subseteq \mathrm{O}(\mathbb{H}_X)$ sending an automorphism $g$ to its induced action $g^*$ on $\Num(X)$. The homomorphism $\varphi$ has finite kernel by \cite[Proposition 2.1]{DolgachevMartinNUM}, so we can identify $\Aut(X)$ with the discrete group of isometries $\Gamma \coloneqq \varphi(\Aut(X))$ up to a finite group. Recall that a discrete group $G$ of isometries of $\mathbb{H}_X$ is \emph{elementary} if it has a finite orbit in the closure $\overline{\mathbb{H}}_X$ \cite[§5.5]{ratcliffe}. The group $G$ is \emph{elementary of elliptic type} if it is finite, \emph{elementary of parabolic type} if it fixes a unique boundary point of $\overline{\mathbb{H}}_X$, and \emph{elementary of hyperbolic type} otherwise. Note that, if $H$ is a subgroup of $G$ of finite index, then $H$ is elementary if and only if $G$ is elementary of the same type.

(1) $\Rightarrow$ (2): If $X$ has zero entropy, then all isometries in $\Gamma$ are either elliptic or parabolic. Hence $\Gamma$ is elementary by \cite[Theorem~12.2.3]{ratcliffe}, and thus virtually abelian by \cite[Theorem~5.5.9]{ratcliffe}.

(2) $\Rightarrow$ (3): This is clear.

(3) $\Rightarrow$ (4): Let $\Gamma'$ be a solvable subgroup of $\Gamma$ of finite index. By \cite[Theorem~5.5.10]{ratcliffe}, $\Gamma'$ is elementary, and since by assumption $\Gamma'$ is infinite, $\Gamma'$ is either elementary of parabolic or hyperbolic type.
We claim that $\Gamma'$ (and thus $\Gamma$) is elementary of parabolic type. Seeking a contradiction, assume that it is of hyperbolic type. Then by \cite[Theorem~5.5.8]{ratcliffe}, every element of infinite order of $\Gamma'$ is hyperbolic, and thus it has positive entropy. This is a contradiction, because by \Cref{prop:Enriques.finite.Aut} there exists at least one non-extremal genus~\(1\) fibration $|2F|$ on~\(X\), which induces a parabolic element of $\Gamma'$ of infinite order.
Therefore $\Gamma$ is of parabolic type, and it fixes a unique boundary point of $\mathbb{H}_X$, namely the point corresponding to the class of~\(F\) in $\Num(X)$. Let $|2F_1|$ be any genus~\(1\) fibration of $X$ different from $|2F|$. The subgroup $\varphi(\MW(|2F_1|))$ of $\Gamma$ is elementary and it fixes at least two distinct points in the boundary of $\overline{\mathbb{H}}_X$, corresponding to $F$ and $F_1$, and therefore it is elementary of elliptic type, hence finite.

(4) $\Rightarrow$ (5): The unique non-extremal genus~\(1\) fibration on $X$ is preserved by all of $\Aut(X)$.

(5) $\Rightarrow$ (1): Every automorphism of $X$ preserves a genus~\(1\) fibration, hence it preserves the class of a half-fiber $F$, which induces a nef isotropic class in $\Num(X)$. Thus, $X$ has zero entropy. 
\end{proof}

\begin{remark}
The implications (2) \(\Rightarrow\) (3) and (4) \(\Rightarrow\) (5) \(\Rightarrow\) (1) in \Cref{prop:Enriques.zero.entropy} hold for every surface $X$. The implication (1) $\Rightarrow$ (2) holds for any surface \(X\) such that the natural homomorphism \(\Aut(X) \to \mathrm{O}(\Num(X))\) has finite kernel. Moreover, the implication (3) $\Rightarrow$ (4) holds if one further assumes that \(X\) has a non-extremal genus~\(1\) fibration.

In particular, the proof of the implication (3) $\Rightarrow$ (4) fails for K3 surfaces. And indeed, there exist K3 surfaces of positive entropy with virtually cyclic automorphism group (see \cite[Remark~3.10]{Brandhorst.Mezzedimi:K3.zero_entropy}). 
The main difference, compared with Enriques surfaces, is that all genus $1$ fibrations on these K3 surfaces have finite Mordell--Weil group. So, \Cref{prop:Enriques.zero.entropy} fails for K3 surfaces because \Cref{prop:Enriques.finite.Aut} does. 
\end{remark}

\section{Examples} \label{sec: examples}
The goal of this section is to show that the Enriques surfaces appearing in \Cref{thm: main} have zero entropy. Along the way, we compute their automorphism groups and number of moduli.

\subsection{Type \texorpdfstring{\(\AA\)}{A7}} 
Given an Enriques surface of type \(\AA\), we let \(F_0\) be the (multiplicative) half-fiber of type \(\I_8\) which can be found in the defining dual graph:
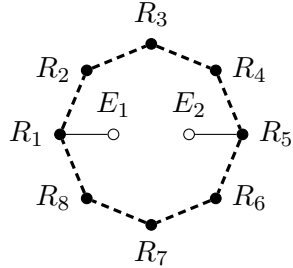
\[
    \begin{tikzpicture}[scale=0.6]
        \node (R1) at (180:2) [nodal,label=left:\(R_1\)] {};
        \node (R2) at (135:2) [nodal,label=left:\(R_2\)] {};
        \node (R3) at (90:2) [nodal,label=above:\(R_3\)] {};
        \node (R4) at (45:2) [nodal,label=right:\(R_4\)] {};
        \node (R5) at (0:2) [nodal,label=right:\(R_5\)] {};
        \node (R6) at (315:2) [nodal,label=right:\(R_6\)] {};
        \node (R7) at (270:2) [nodal,label=below:\(R_7\)] {};
        \node (R8) at (225:2) [nodal,label=left:\(R_8\)] {};
        \node (R9) at (intersection of R2--R7 and R3--R8) [nodal, fill=white, label=above:\(E_1\)] {};
        \node (R10) at (intersection of R4--R7 and R3--R6) [nodal, fill=white, label=above:\(E_2\)] {};
        \draw[densely dashed, very thick] (R1)--(R2)--(R3)--(R4)--(R5)--(R6)--(R7)--(R8)--(R1);
        \draw (R1)--(R9) (R5)--(R10);
    \end{tikzpicture}
\]
Note that the divisor $F_0$ is indeed a half-fiber, since $E_1.F_0 = 1$, so $F_0$ is primitive in $\Pic(X)$.
By \Cref{lem: genus.1.fibrations}, Enriques surfaces of type \(\AA\) in characteristic \(2\) are ordinary.

\begin{lemma} \label{lem: num.triv.of.BP}
Let $X$ be an Enriques surface of type \(\AA\). Then, $X$ admits a unique numerically trivial involution $\sigma$. 
More precisely, $\sigma \in \MW(|2F_0|)$ and the reduced divisorial part $X_1^{\sigma}$ of the fixed locus of $\sigma$ is
\[
X_1^{\sigma} = 
\begin{cases}
F_0 &\text{ if } p = 2,  \\
F_0' + \sum_{i=0}^3 R_{2i+1} & \text{ if } p \neq 2,
\end{cases}
\]
where $F_0'$ is the second half-fiber of $|2F_0|$. Moreover, $F_0'$ is smooth.
\end{lemma}
\begin{proof}
Since $X$ is ordinary if $p = 2$, the K3 cover $\pi \colon \widetilde{X} \to X$ is étale with covering involution $\tau$. The preimage $\pi^{-1}(F_0)$ of the half-fiber $F_0$ of type~$\I_8$ is a fiber of type~$\I_{16}$ of an elliptic fibration on $\widetilde{X}$ and the preimages of the two bisections $E_1,E_2$ of $|2F_0|$ are four disjoint sections $E_1^{\pm},E_2^{\pm}$ of the induced elliptic fibration $\widetilde{f} \colon \widetilde{X} \to \mathbb{P}^1$. Choosing $E_1^+$ as the zero section, the section $E_1^{-}$ is a $2$-torsion section by the height pairing \cite[Section 6.5]{ShiodaSchuett} and the quotient of $\widetilde{X}$ by the involution $\tau'$ obtained by composing the translation by $E_1^{-}$ with $\tau$ is birational to $J(X)$ by \cite[Lemma 2.17]{MartinFiniteAut}.
The induced generically finite morphism $\pi' \colon \widetilde{X} \to J(X)$ sends the four sections $E_1^{\pm},E_2^{\pm}$ to four sections of~$J(f)$.

Now, we take $\pi'(E_1^-) \in \MW(J(f))$ and let $\sigma$ be the induced involution of~$X$. In other words, $\sigma$ is the automorphism of $X$ induced by $\tau'$ via $\pi$.
We describe the divisorial part of the fixed locus of $\sigma$. By construction, $\sigma$ acts as a translation on the simple fibers of $f$, hence the divisorial part of the fixed locus of $\sigma$ is contained in the half-fibers $F_0$ and $F_0'$ (where $F_0'$ only exists if $p \neq 2$).

To understand the action of $\sigma$ on $F_0$, note that $\tau'$ preserves all components of $\pi^{-1}(F_0)$. If $p = 2$, the fact that involutions of $\mathbb{P}^1$ have only one fixed point implies that $\tau'$ fixes every component of $\pi^{-1}(F_0)$ pointwise, hence $\sigma$ fixes $F_0$ pointwise. If $p \neq 2$, then $\tau'$ is anti-symplectic, so its fixed locus has pure codimension $1$. Thus, the smoothness of fixed loci of tame involutions implies that every other component of $\pi^{-1}(F_0)$ is fixed pointwise. Therefore, $\sigma$ fixes $R_1,R_3,R_5,$ and $R_7$ pointwise.


For the action of $\sigma$ on $F_0'$, first note that $p \neq 2$ in this case. Then, the automorphism $\tau'$ of the previous paragraph fixes the point $F_0' \cap E_1^+$, hence, again because $\tau'$ is anti-symplectic, it fixes $F_0'$ pointwise. The smoothness of fixed loci of tame involutions implies that $F_0'$ is smooth.

In particular, we see that, in all characteristics, $\sigma$ preserves all components of the defining graph of $X$ and it is easy to check that $\Num(X) \otimes \mathbb{Q}$ is generated by the curves in this graph, hence $\sigma$ is numerically trivial. The uniqueness of $\sigma$ follows from \cite[Theorem on p.~1182]{DolgachevMartinNUM}.
\end{proof}

\begin{remark}
In fact, $\sigma$ is even cohomologically trivial. This is clear if $p = 2$, for then $K_X \sim 0$, and if $p = 0$, this is proved in \cite[Proposition 4.8]{BarthPeters}. Via specialization, this implies cohomological triviality of $\sigma$ also in odd characteristic.
\end{remark}

\begin{theorem} \label{thm:AA}
Every Enriques surface of type \(\AA\) has zero entropy. More precisely, \(|2F_0|\) is the unique non-extremal genus \(1\) fibration on \(X\) if \(\Aut(X)\) is infinite.
\end{theorem}
\begin{proof}
By \Cref{lem: num.triv.of.BP}, $X$ admits a unique numerically trivial involution $\sigma$. Since $\sigma$ is unique, the subgroup generated by $\sigma$ is normal and hence central in $\Aut(X)$. Thus, $\Aut(X)$ preserves the fixed locus of $\sigma$. By \Cref{lem: num.triv.of.BP}, the fibration $|2F_0|$ is the unique genus one fibration with a half-fiber contained in the fixed locus of $\sigma$, hence $\Aut(X)$ preserves $|2F_0|$ and so $X$ has zero entropy by \Cref{prop:Enriques.zero.entropy}.
\end{proof}

Recall that, by the proof of \Cref{lem: num.triv.of.BP}, the Jacobian $J(f)$ of $|2F_0|$
admits a $2$-torsion section. It is well-known that torsion sections on rational elliptic surfaces are disjoint from the zero section \cite[Proposition 5.4]{OguisoShioda}. Thus, if $p = 2$, then $J(f)$ admits no irreducible multiplicative fibers (because $\mathbb{G}_m$ admits no point of order $2$ in this case), and if $p \neq 2$, then $J(f)$ admits no irreducible additive fibers (because $\mathbb{G}_a$ admits no point of order $2$ in this case).
If $\Aut(X)$ is infinite, $J(f)$ has infinite Mordell--Weil group by \Cref{thm:AA}, so the $\I_8$-fiber is its only reducible fiber. From \cite{LangConfigs}, we conclude that if $p = 2$, then the singular fibers of $J(f)$ are of type~$\I_8$ and $\II$, and if $p \neq 2$, they are of type~$\I_8$ and $\I_1,\I_1,\I_1,\I_1$. In the following, we let $D_{\infty} = \mathbb{Z} \rtimes \mathbb{Z}/2\mathbb{Z}$ be the infinite dihedral group.



\begin{proposition} \label{prop: Aut of BP}
Let $X$ be an Enriques surface of type \(\AA\) with infinite automorphism group.
Then, the following hold:
\begin{enumerate}
\item If $p = 2$, then $\Aut(X)  \cong \mathbb{Z}/2\mathbb{Z} \times D_{\infty}$.
\item If $p \neq 2$, assume that \(F_0\) and \(F_0'\) lie over $[0:1]$ and $[1:0]$ and let $p_i = [a_i:1]$ with $a_i \in k^{\times}$ be the images of the four nodal fibers of~\(|2F_0|\).
Set $a_1 = 1$. Then,
\begin{enumerate}
\item if $ \{a_1,a_2,a_3,a_4\} = \{1,\zeta_4,\zeta_4^2,\zeta_4^3\}$ for a primitive $4$-th root of unity $\zeta_4$, then $\Aut(X)$ is a non-split extension of $\mathbb{Z}/2\mathbb{Z}$ by $\mathbb{Z}/4\mathbb{Z} \times D_{\infty}$.
\item if there exists $a \in k$ with $a^4 \neq 1$ and $\{a_1,a_2,a_3,a_4\} = \{1,-1,a,-a\}$, then $\Aut(X) \cong \mathbb{Z}/4\mathbb{Z} \times D_{\infty}$.
\item if (a) and (b) do not hold, then $\Aut(X) \cong \mathbb{Z}/2\mathbb{Z} \times D_{\infty}$.
\end{enumerate}
\end{enumerate}

\end{proposition}
\begin{proof}
Let \(f\colon X \to \IP^1\) be the elliptic fibration induced by \(|2F_0|\).
By \Cref{lem: num.triv.of.BP}, we know that $\MW(J(f))$ contains an element of order $2$, hence, by \cite[Main Theorem]{OguisoShioda}, we have $\MW(J(f)) \cong \mathbb{Z}/2\mathbb{Z} \times \mathbb{Z}$. 
Since $f$ is not isotrivial, because there is a multiplicative fiber, we have $\Aut_{\mathbb{P}^1}(X) \cong \mathbb{Z}/2\mathbb{Z} \times D_{\infty}$ by \Cref{lem: sign involution}.

It remains to study the image of the homomorphism $\rho \colon \Aut(X) \to {\rm PGL}_2$ induced by the action of $\Aut(X)$ on the base of~\(f\). Let $g \in \Aut(X)$. Then, by functoriality of the Jacobian, $g$ induces an automorphism $g'$ of $J(X)$ that acts on the base of $J(f)$ as $\rho(g)$. By \cite[Theorem~3.3]{DolgachevMartin}, $g'$ preserves the zero section, hence also the unique $2$-torsion section of $J(f)$.

If $p \neq 2$, an explicit computation with Weierstraß equations using Tate's algorithm \cite{Tate} shows that $J(f)$ can be defined by an equation of the form
\begin{equation} \label{eq: Jacobian.of.BP}
y^2 = x^3 + 2a_2(s,t)x^2 + t^4x,
\end{equation}
where the \(2\)-torsion section is given by \((x,y) = (0,0)\), the $\I_8$-fiber lies over $t = 0$ and the other singular fibers lie over the roots of $\Delta_0(s,t) = a_2^2 - t^4$, where $\Delta(s,t) = - 64 t^8 \Delta_0$ is the discriminant. Since we assume that $X$ has infinite automorphism group, the four roots of $\Delta_0$ must be distinct.
We know from \Cref{lem: num.triv.of.BP} that the other half-fiber of $f$ is smooth, and after a change of coordinates, we may assume that it lies over $s=0$. In particular, both $\Delta_0(0,1)$ and $\Delta_0(1,0)$ are non-zero.
Since $\rho(g)$ fixes the two points corresponding to the half-fibers of~\(f\), we know that $\rho(g)$ fixes the points $t = 0$ and $s = 0$, so that it is given by an automorphism of the form $s \mapsto \lambda s$ for some $\lambda \in k^{\times}$. 

Next, write $a_2 = as^2 + bst + ct^2$ with $a,b,c \in k$ and recall that $\rho(g)$ preserves the roots of 
\[
\Delta_0 = a_2^2 - t^4 = a^2s^4 + 2abs^3t + (2ac + b^2)s^2t^2 + 2bcst^3 + (c^2 - 1)t^4
\]
and that $c^2 - 1 = \Delta_0(0,1) \neq 0$ and $a^2 = \Delta_0(1,0) \neq 0$.
In particular, we can rescale coordinates to assume $\Delta_0(1,1) = 0$.
Since $\rho(g)$ preserves the roots of $\Delta_0$, the polynomial $\Delta_0(\lambda s ,t)$ must be a multiple of $\Delta_0(s,t)$.
The automorphism $\rho(g)$ changes the coefficients of the monomials in $\Delta_0(s,t)$ as follows:
\[
[a^2:2ab:2ac+b^2:2bc:c^2-1] \mapsto [\lambda^4 a^2:2\lambda^3ab:\lambda^2(2ac+b^2):\lambda(2bc):c^2-1].
\] 
We deduce that $\lambda^4 = 1$, that $b = 0$ if $\lambda \neq 1$, and that $b = c = 0$ if $\lambda^2 \neq 1$, i.e., the order of \(\rho(g)\) is \(n \in \{1,2,4\}\).

Observe that, in every case, $\rho(g)$ sends $a_2$ to $\lambda^2 a_2$ and that $\lambda^4 = 1$. Hence, up to composing with $y \mapsto - y$, we deduce from the structure of isomorphisms between Weierstraß forms (see, e.g., \cite[Chapter~III, Proposition~3.1]{Silverman}) that $g'$ must be an automorphism of the following form:
\[
g' \colon (s,t,x,y) \mapsto (\lambda s ,t,\lambda^2 x,\lambda y).
\]
Now, we reverse the construction of \Cref{lem: num.triv.of.BP}: the Weierstraß model of the K3 cover $\widetilde{X}$ of $X$ is given by replacing $(s,t)$ by $(s^2,t^2)$ in Equation \eqref{eq: Jacobian.of.BP}, hence by
\[
y^2 = x^3 + 2a_2(s^2,t^2)x^2 + t^8x,
\]
and the Enriques involution $\tau$ is the composition of $(s,t) \mapsto (-s,t)$ with the translation by the section with $(x,y) = (0,0)$. The automorphism $g'$ lifts to $\widetilde{X}$ as 
\[
\widetilde{g'} \colon (s,t,x,y) \mapsto (\sqrt{\lambda} s,t,\lambda^2 x, \lambda y).
\]
This is an automorphism of order $2n$, where $n$ is the order of $\lambda \in k^{\times}$, and it commutes with $\tau$, hence it yields an automorphism $g'' \in \Aut(X)$ of order $2n$. Note that $g''^n = \sigma$ is the cohomologically trivial involution of $X$. Indeed, $\widetilde{g'}^n \circ \tau$ coincides with the translation by $(x,y) = (0,0)$, and this induces $\sigma$ on~\(X\) by \Cref{lem: num.triv.of.BP}.

We conclude that $\Aut(X)$ is generated by a $g''$ as above with maximal $n$, an involution $\iota$ of the generic fiber of~\(f\) as in \Cref{lem: sign involution}, and the translation by a generator $E$ of the free part of $\MW(J(f))$. We now have three cases:
\begin{enumerate}
\item $n = 1$: This happens if and only if $b \neq 0$. In this case, the action of $\Aut(X)$ on the base of~\(f\) is trivial and $\Aut(X) = \Aut_{\mathbb{P}^1}(X) \cong \mathbb{Z}/2\mathbb{Z} \times D_{\infty}$. Note that in this case $\Delta_0(-1,1) \neq 0$, so this corresponds to Case (2) (c) in the statement of the proposition.
\item $n = 2$: This happens if and only if $b = 0$ and $c \neq 0$. In this case, pick $d \in k$ with $d^2 + 2cd + 1 = 0$ and let $E$ be the section $(x,y) = (dt^2, \sqrt{2a}dst^2)$. This is a section of height $\frac{1}{2}$, hence a generator of the free part of $\MW(J(f))$ by \cite{OguisoShioda}. The automorphism $\widetilde{g'}$ preserves the preimage $(x,y) = (dt^4,\sqrt{2a}ds^2t^4)$ of $E$ in $\widetilde{X}$, hence commutes with the translation by $E$, and hence so does $g''$. As $g''$ commutes with $\iota$, we deduce that $\Aut(X) \cong \mathbb{Z}/4\mathbb{Z} \times D_{\infty}$. Note that in this case $\Delta_0(-1,1) = \Delta_0(1,1) = 0$ but $\Delta_0(\zeta_4,1) \neq 0$, so this corresponds to Case (2) (b) in the statement of the proposition.
\item $n = 4$:  This happens if and only if $b = c = 0$. Here, $\Aut(X)$ contains the automorphism group of Case (2) as a normal subgroup of index $2$, hence $\Aut(X)$ is an extension of $\mathbb{Z}/2\mathbb{Z}$ by $\mathbb{Z}/4\mathbb{Z} \times D_{\infty}$. This extension does not split, since $\Aut(X)$ acts through $\mathbb{Z}/4\mathbb{Z}$ on $\mathbb{P}^1$, while $\mathbb{Z}/4\mathbb{Z} \times D_{\infty}$ acts through $\mathbb{Z}/2\mathbb{Z}$. Note that in this case $\Delta_0(\zeta_4^i,1) = 0$ for $i = 0,1,2,3$, so this corresponds to Case (2) (a) in the statement of the proposition.

\end{enumerate}

Next, assume that $p = 2$. By \cite[Section 2]{LangConfigs}, the $j$-map of $J(f)$ has degree $8$ and it is $\rho(\Aut(X))$-invariant, hence the order of $\rho(\Aut(X))$ is a power of $2$. On the other hand, $\rho(\Aut(X))$ fixes the two points corresponding to the two half-fibers, hence it must have odd order. We conclude that $\rho(\Aut(X))$ must be trivial if $p = 2$, hence $\Aut(X) = \Aut_{\mathbb{P}^1}(X) \cong \mathbb{Z}/2\mathbb{Z} \times D_{\infty}$.
\end{proof}

\begin{remark}
The proof of \Cref{prop: Aut of BP} shows that Enriques surfaces of type \(\tilde A_7\) form a $2$-dimensional family in all characteristics. If $p \neq 2$, the subfamily where $\Aut(X) \cong \mathbb{Z}/4\mathbb{Z} \times D_{\infty}$ is $1$-dimensional and there is a unique Enriques surface of type \(\AA\) where $\Aut(X)$ is a non-split extension of $\mathbb{Z}/2\mathbb{Z}$ by $\mathbb{Z}/4\mathbb{Z} \times D_{\infty}$. We also recall that there is a $1$-dimensional family of Enriques surfaces of type~$\widetilde{A}_7$ with finite automorphism group in all characteristics. These are called ``type~$\I$'' surfaces in \cite{Kondo} and \cite{MartinFiniteAut}.
\end{remark}

\begin{remark} \label{remark:BP.wrong}
In \cite[Theorem 4.12]{BarthPeters}, it is claimed that the automorphism group of a surface of type \(\tilde A_7\) is never larger than $\mathbb{Z}/4\mathbb{Z} \times D_{\infty}$. This is due to an erroneous calculation in the last lines of the proof of \cite[Lemma 4.3]{BarthPeters}. 

The Enriques surface of type \(\tilde A_7\) that corresponds to Case (2) (a) in \Cref{prop: Aut of BP} can be realized as a member of the family considered by Barth--Peters as follows.
Consider the double cover $\widetilde{X}$ of $\mathbb{P}^1 \times \mathbb{P}^1$ branched over the curve
\[
(v_0^2 - v_1^2)((v_0^2-v_1^2)u_0^4 + (v_0^2 + v_1^2)u_1^4)).
\]
This corresponds to the parameters $(a,b,c,d) = (1,0,1,-1)$ in \cite[Section 4.1]{BarthPeters}. The minimal resolution of the quotient of $\widetilde{X}$ by the involution induced by 
\[([u_0:u_1],[v_0:v_1]) \mapsto ([-u_0:u_1],[-v_0:v_1])\] 
is an Enriques surface of type \(\tilde A_7\). This surface admits an automorphism of order $8$ induced by 
\[([u_0:u_1],[v_0:v_1]) \mapsto ([\zeta_8 u_0:u_1],[\zeta_4 v_1: \zeta_4 v_0]).\]
The existence of this automorphism contradicts the last two lines of the proof of \cite[Lemma 4.3]{BarthPeters}.

This model can be used to give a more explicit description of $\Aut(X)$. We leave the details to the interested reader.
\end{remark}

\subsection{Type \texorpdfstring{\(\EE\)}{E6}}
Given an Enriques surface of type \(\EE\), we let \(F_0\) be the (additive) half-fiber of type \(\IV^*\) which can be found in the defining dual graph:
\[
 \begin{tikzpicture}
        \node (R0) at (90:1.5) [nodal,fill=white] {};
        \node (R1) at (90:1) [nodal] {};
        \node (R2) at (90:0.5) [nodal] {};
        \node (R3) at (0:0) [nodal] {};
        \node (R4) at (210:0.5) [nodal] {};
        \node (R5) at (210:1) [nodal] {};
        \node (R6) at (210:1.5) [nodal,fill=white] {};
        \node (R7) at (330:0.5) [nodal] {};
        \node (R8) at (330:1) [nodal] {};
        \node (R9) at (330:1.5) [nodal,fill=white] {};
        \draw[densely dashed, very thick] (R1)--(R2)--(R3)--(R4)--(R5) (R3)--(R7)--(R8);
        \draw (R0)--(R1) (R5)--(R6) (R8)--(R9);
    \end{tikzpicture}
\]
By \Cref{lem: genus.1.fibrations}, Enriques surfaces of type \(\EE\) only exist in characteristic \(2\) and are either classical or supersingular.

Recall that classical and supersingular Enriques surfaces $X$ in characteristic $2$ have the property that $h^0(X,\Omega_X) = 1$. The divisorial part $D$ of the zero locus~$Z$ of a global $1$-form on~\(X\) is called \emph{conductrix}. For general $X$, this conductrix is empty and $Z$ consists of $12$ reduced points by the Hirzebruch--Riemann--Roch formula. In contrast, if $X$ is of type \(\EE\), then \(D = F_0\) by \cite{Ekedahl.Shepherd-Barron}. It is clear that all automorphisms of $X$ preserve the conductrix, so we deduce the following theorem from \Cref{prop:Enriques.zero.entropy}.

\begin{theorem} \label{thm:EE}
Every Enriques surface of type \(\EE\) has zero entropy. More precisely, \(|2F_0|\) is the unique non-extremal genus \(1\) fibration on \(X\) if \(\Aut(X)\) is infinite.
\end{theorem}

\begin{proposition}
If $X$ is an Enriques surface of type \(\EE\) with infinite automorphism group, then 
\[
\Aut(X) \cong \MW(|2F_0|) \rtimes \mathbb{Z}/2\mathbb{Z} \quad \text{and} \quad \MW(|2F_0|) \in \{\mathbb{Z},\mathbb{Z}^2\}.
\]
\end{proposition}

\begin{proof}
Let \(f\colon X \to \IP^1\) be the elliptic fibration induced by \(|2F_0|\).
First, assume that $f$ is non-isotrivial, so that it admits a fiber $G$ of type~$\I_n$ with $n \geq 1$.
By \cite[Lemma~3.2]{Katsura.Kondo.Martin}, the canonical cover $\pi \colon \widetilde{X} \to X$ coincides with the Frobenius pullback of $J(f)$ in a neighborhood of $G$, hence $\widetilde{X}$ has some $A_1$-singularities over~$G$, so by \cite[Corollary~3.7]{Katsura.Kondo.Martin} and because $X$ admits a quasi-elliptic fibration by \cite[Theorem~1.1]{Ekedahl.Shepherd-Barron}, $X$ is classical.

The second half-fiber $F'_0$ of~\(f\) is not of type~$\I_n$ by \Cref{lem: halffibers} and not of the same type as $F_0$, because the Picard rank of $X$ is $10$. Since $\Aut(X)$ preserves both $F_0$ and $F'_0$, it acts on $\mathbb{P}^1$ through $k^{\times}$. Thus, if this action is non-trivial, then the number of fibers of type~$\I_n$ for a given $n$ must be odd, contradicting the possible configurations of singular fibers of~\(f\) determined in \cite[p.~5826]{LangConfigs}. 
Therefore, $\Aut(X)$ coincides with the automorphism group of the generic fiber $F_{\eta}$ of~\(f\), which is $\MW(J(f)) \rtimes \mathbb{Z}/2\mathbb{Z}$ by \Cref{lem: sign involution}, since $F_{\eta}$ is ordinary. By \cite{OguisoShioda}, we know that $\MW(J(f))$ is either $\mathbb{Z}$ or $\mathbb{Z}^2$, depending on whether $f$ has a second reducible fiber or not.

Next, assume that $f$ is isotrivial. Then, by \cite[p.~5834]{LangConfigs}, $f$ admits a second singular fiber $G$ of type~$\II$ or $\III$. 
A computation using Tate's algorithm shows that $J(f)$ admits a Weierstraß equation of the form
\[
y^2 + st^2y = x^3 + at^2x^2 + bt^6,
\]
with $a,b \in k$, not both $0$. The fibration $J(f)$ admits a fiber of type~$\IV^*$ over $t = 0$ and the other singular fiber $G$ over $s = 0$.
Every automorphism of $J(X)$ is of the form
\[
g'\colon (s,t,x,y) \mapsto (\lambda s, \mu t, \beta x + b_2(s,t), y + b_1(s,t)x + b_3(s,t))
\]
which sends the Weierstraß form to
\begin{multline*}
    y^2 + \lambda \mu^2 st^2 y = \beta^3 x^3 + (\beta^2 b_2 + b_1^2 + a \beta^2 \mu^2 t^2) x^2 \\ + (\lambda \mu^2 st^2 b_1 + \beta b_2^2) x + b_3^2 + \lambda \mu^2  st^2b_3 + b_2^3 + a \mu^2 t^2b_2^2 + b \mu^6 t^6.
\end{multline*}
Comparing coefficients of $x$, we see that $s \mid b_1$ and $st \mid b_2$, and then comparing the coefficients of $x^2$ yields $b_1 = b_2 = 0$. Moreover, $\lambda = \mu^{-2}$ and $\beta^3 = 1$, so that the above new Weierstraß equation simplifies to
$$
y^2 + st^2 y = x^3 + a \beta^2 \mu^2 t^2 x^2 + b_3^2 + st^2b_3 + b \mu^6 t^6.
$$
Thus, if $a \neq 0$, then $\beta^2\mu^2 = 1$, hence $\mu$ is a third root of unity and $b_3 \in \{0,st^2\}$, so that $g'$, if non-trivial, is the sign involution.

If $a = 0$, we can rescale coordinates to assume $b = 1$. Then, the only additional condition we have is that $b_3^2 + st^2 b_3 = (1 + \mu^6) t^6$. If $t^6$ occurs on the left-hand side with non-zero coefficient, then so does $st^5$, which is absurd. Hence, $\mu^6 = 1$. But then again $\mu^3 = 1$ and $b_3 \in \{0,st^2\}$, so that $g'$ is either trivial or the sign involution.

This shows that $\Aut(J(X)) = \MW(J(f)) \rtimes \mathbb{Z}/2\mathbb{Z}$ acting trivially on the base of $J(f)$. Thus, $\Aut(X)$ acts trivially on the base of~\(f\) and if $F_{\eta}$ is the generic fiber of~\(f\), then the natural map $\varphi\colon \Aut(F_{\eta}) \to \Aut(\Pic^0_{F_\eta})$, whose kernel is $\MW(J(f))$, factors through $\mathbb{Z}/2\mathbb{Z}$. The sign involution on the generic fiber of~\(f\) has non-zero image under $\varphi$, so the proposition follows.
\end{proof}

\begin{remark}
By \cite[Theorem~1]{Salomonsson}, Enriques surfaces of type \(\EE\) form a \(3\)-di\-men\-sion\-al family and the generic member satisfies $\MW(|2F_0|) \cong \mathbb{Z}^2$.
The surfaces with $\MW(|2F_0|) \cong \mathbb{Z}$ form a subfamily of dimension $2$ and the ones with finite automorphism group form a subfamily of dimension $1$. In each of these strata, the generic member is a classical Enriques surface and the supersingular Enriques surfaces form a subfamily of codimension $1$.
\end{remark}

\subsection{Type \texorpdfstring{\(\XX\)}{D6+A1}}

Given an Enriques surface of type \(\XX\), we let \(F_1\) be the following (additive) half-fiber of type \(\I_2^*\) which can be found in the defining dual graph:
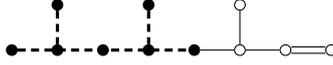
\[
 \begin{tikzpicture}[scale=0.6]
    \node (R4) at (0,0) [nodal] {};
    \node (R5) at (1,1) [nodal] {};
    \node (R6) at (1,0) [nodal] {};
    \node (R7) at (2,0) [nodal] {};
    \node (R8) at (3,0) [nodal,fill=white] {};
    \node (R9) at (4,0) [nodal,fill=white] {};
    \node (R3) at (-1,0) [nodal] {};
    \node (R2) at (-2,0) [nodal] {};
    \node (R1) at (3,1) [nodal,fill=white] {};
    \node (RX) at (-1,1) [nodal] {};
    \node (RXX) at (5,0) [nodal,fill=white] {};
    \draw (R7)--(R8)--(R9) (R8)--(R1);
    \draw[densely dashed, very thick] (R2)--(R7) (R5)--(R6) (RX)--(R3);
    \draw[double] (R9)--(RXX);
\end{tikzpicture}
\]


\begin{lemma} \label{lem: xxgeneralities}
The following hold:
\begin{enumerate}
\item An Enriques surface is of type~$\XX$ if and only if it admits a genus~\(1\) fibration with half-fibers of type~$\I_2^*$ and $\III$. This fibration is quasi-elliptic.
\item Every Enriques surface of type~$\XX$ is a classical Enriques surface in characteristic $2$.
\item Enriques surfaces of type~$\XX$ exist and form a $2$-dimensional family.
\item The conductrix of an Enriques surface of type~$\XX$ looks as follows:
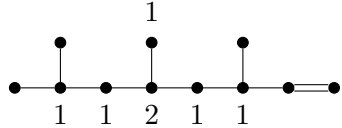
\begin{center}
    \begin{tabular}{cc}
 \begin{tikzpicture}[scale=0.6]
    \node (R4) at (0,0) [nodal, label=below:$1$] {};
    \node (R5) at (1,1) [nodal, label=above:$1$] {};
    \node (R6) at (1,0) [nodal, label=below:$2$] {};
    \node (R7) at (2,0) [nodal, label=below:$1$] {};
    \node (R8) at (3,0) [nodal, label=below:$1$] {};
    \node (R9) at (4,0) [nodal] {};
    \node (R3) at (-1,0) [nodal, label=below:$1$] {};
    \node (R2) at (-2,0) [nodal] {};
    \node (R1) at (3,1) [nodal] {};
    \node (RX) at (-1,1) [nodal] {};
    \node (RXX) at (5,0) [nodal] {};
\draw (R2)--(R3)--(R6) (R5)--(R6)--(R9) (R1)--(R8) (R3)--(RX);
\draw[double] (R9)--(RXX);
\end{tikzpicture}
    \end{tabular}
    \end{center}
    \item Every Enriques surface of type~$\XX$ admits a unique non-trivial numerically trivial involution $\sigma$. Moreover, \(\sigma \in \MW(|2F_1|)\).
    \item Every Enriques surface of type~$\XX$ has infinite automorphism group.
\end{enumerate}
\end{lemma}
\begin{proof}
For Claim (1), observe from the defining graph of type~$\XX$ that the fibration $|2F_1|$ has a second half-fiber \(F_1'\) of type~$\III$ or $\I_2$, and a third reducible fiber:
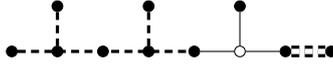
\[
 \begin{tikzpicture}[scale=0.6]
    \node (R4) at (0,0) [nodal] {};
    \node (R5) at (1,1) [nodal] {};
    \node (R6) at (1,0) [nodal] {};
    \node (R7) at (2,0) [nodal] {};
    \node (R8) at (3,0) [nodal,fill=white] {};
    \node (R9) at (4,0) [nodal] {};
    \node (R3) at (-1,0) [nodal] {};
    \node (R2) at (-2,0) [nodal] {};
    \node (R1) at (3,1) [nodal] {};
    \node (RX) at (-1,1) [nodal] {};
    \node (RXX) at (5,0) [nodal] {};
    \draw (R7)--(R8)--(R9) (R8)--(R1);
    \draw[densely dashed, very thick] (R2)--(R7) (R5)--(R6) (RX)--(R3);
    \draw[densely dashed, very thick,double] (R9)--(RXX);
\end{tikzpicture}
\]
In particular, this fibration is extremal. As $X$ admits an additive half-fiber, we are in characteristic $p = 2$ by \Cref{lem: halffibers}, and so \(F_1'\) must be of type~$\III$ .
By \Cref{tab:elliptic}, there exists no extremal rational elliptic fibration with these fibers if $p = 2$, so $f$ must be quasi-elliptic.
Conversely, if an Enriques surface admits a genus $1$ fibration with the given fiber types, then $p = 2$ and the fibration must be quasi-elliptic by \cite[Section 4]{LangConfigs}, so the dual graph of components of fibers and curve of cusps contains the graph of type~$\XX$.

Claim (2) follows from Claim (1) and \Cref{lem: genus.1.fibrations}.

For Claim (3), note that with the alternative description given in Claim (1), surfaces of type~$\XX$ have first been constructed in \cite[Example 7.9]{DolgachevMartin} and the conjectural number of moduli for these surfaces given in \cite[Corollary 7.8]{DolgachevMartin} has recently been confirmed to be $2$ in \cite[Proposition 14.1]{Katsura.Schuett}.

Claim (4) follows from \cite[Table 6]{Katsura.Kondo.Martin}.

For Claim (5), we combine \Cref{cor: actiononsimplefiber} and \Cref{lem: action.on.additive} with \Cref{tab:quasi-elliptic} to deduce that there exists a unique non-trivial $\sigma \in \MW(|2F_1|)$ that preserves all curves in the defining graph of $X$. 
The curves in the graph generate $\Num(X)$ over $\mathbb{Q}$, hence $\sigma$ is numerically trivial. Conversely, if $\sigma'$ is a numerically trivial automorphism of~$X$, then it preserves $f$ and acts trivially on the base, since it preserves the three reducible fibers. If $\sigma'$ had odd order, then, by the known structure of fixed loci of automorphisms of cuspidal curves, the fixed locus of $\sigma'$ would contain an integral curve that has intersection number $1$ with every fiber, which is absurd. Hence, $\sigma'$ has even order, so it must come from $\MW(|2F_1|)$, and so $\sigma' = \sigma$.


Claim (6) follows from the classification of Enriques surfaces with finite automorphism group given in \cite{Katsura.Kondo.Martin}.
\end{proof}

\begin{lemma} \label{lem: xxfixedlocus}
Let $X$ be an Enriques surface of type~$\XX$ and let $\sigma$ be its non-trivial numerically trivial involution. Then, the support of the union of the conductrix of $X$ and the divisorial part of the fixed locus of $\sigma$ forms a configuration \(G_0\) of type~$\tilde{E}_7$:
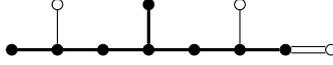
\[
 \begin{tikzpicture}[scale=0.6]
    \node (R4) at (0,0) [nodal] {};
    \node (R5) at (1,1) [nodal] {};
    \node (R6) at (1,0) [nodal] {};
    \node (R7) at (2,0) [nodal] {};
    \node (R8) at (3,0) [nodal] {};
    \node (R9) at (4,0) [nodal] {};
    \node (R3) at (-1,0) [nodal] {};
    \node (R2) at (-2,0) [nodal] {};
    \node (R1) at (3,1) [nodal,fill=white] {};
    \node (RX) at (-1,1) [nodal,fill=white] {};
    \node (RXX) at (5,0) [nodal,fill=white] {};
    \draw (R1)--(R8) (R3)--(RX);
    \draw[very thick] (R2)--(R9) (R5)--(R6);
    \draw[double] (R9)--(RXX);
\end{tikzpicture}
\]
\end{lemma}
\begin{proof}
Since $\sigma$ is numerically trivial, it preserves all curves that appear in the defining graph of $X$. Since non-trivial involutions on $\mathbb{P}^1$ in characteristic $2$ have a unique fixed point, we deduce that $\sigma$ fixes pointwise the curves corresponding to the black vertices in the following graph:

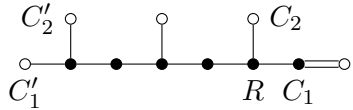
\begin{center}
    \begin{tabular}{cc}
 \begin{tikzpicture}[scale=0.6]
    \node (R4) at (0,0) [nodal] {};
    \node (R5) at (1,1) [nodal,fill=white] {};
    \node (R6) at (1,0) [nodal] {};
    \node (R7) at (2,0) [nodal] {};
    \node (R8) at (3,0) [nodal, label=below:$R$] {};
    \node (R9) at (4,0) [nodal, label=below:$C_1$] {};
    \node (R3) at (-1,0) [nodal] {};
    \node (R2) at (-2,0) [nodal,fill=white, label=below:$C_1'$] {};
    \node (R1) at (3,1) [nodal,fill=white, label=right:$C_2$] {};
    \node (RX) at (-1,1) [nodal,fill=white, label=left:$C_2'$] {};
    \node (RXX) at (5,0) [nodal,fill=white] {};
\draw (R2)--(R3)--(R6) (R5)--(R6)--(R9) (R1)--(R8) (R3)--(RX);
\draw[double] (R9)--(RXX);
\end{tikzpicture}
    \end{tabular}
    \end{center}

By \Cref{cor: actiononsimplefiber}, the Mordell--Weil group of the genus~\(1\) fibration $|2F_2|$ with fiber of type~$\I_4^*$ in the defining diagram induces a horizontal reflection on the graph obtained by removing the right-most vertex, so we may assume without loss of generality that $C_1$ and $C_1'$ resp. $C_2$ and $C_2'$ are interchanged by this involution. Since $\sigma$ is the only non-trivial numerically trivial automorphism, the whole automorphism group $\Aut(X)$ commutes with $\sigma$, and hence preserves the fixed locus of $\sigma$. We conclude that, as $\sigma$ fixes $C_1$ pointwise, it must also fix $C_1'$ pointwise.

It remains to show that the remaining components of the fixed locus of $\sigma$ lie in the conductrix. Since \(\sigma \in \MW(|2F_1|)\) by \Cref{lem: xxgeneralities}, the fixed locus of $\sigma$ is contained in the union of the curve of cusps $R$ and fibers of~\(|2F_1|\). By \cite{Ito:char2}, the sections of the Jacobian of~\(|2F_1|\) are disjoint, hence by \Cref{cor: actiononsimplefiber}, every fixed point of $\sigma$ on a simple fiber of~\(|2F_1|\) lies on $R$. We conclude that $C_2$ is not in $X^{\sigma}$, so neither is $C_2'$. Finally, because $X^{\sigma}$ is stable under the involution in $\MW(|2F_2|)$ described in the previous paragraph, the right-most vertex of the defining graph of~$X$ is also not in $X^{\sigma}$. This finishes the proof.
\end{proof}

\begin{proposition} \label{prop: xximportantproperties}
Let $X$ be an Enriques surface of type~$\XX$ and let $G_0$ the configuration of type~$\tilde{E}_7$ on~\(X\) described in \Cref{lem: xxfixedlocus}. Let $F_0$ be a half-fiber with $G_0 \in |2F_0|$ and let $\sigma \in \Aut(X)$ be the numerically trivial involution.
Then, the following hold:
\begin{enumerate}
    \item The fiber $G_0$ is preserved by all of $\Aut(X)$.
    \item The involution $\sigma$ exchanges the two half-fibers of $|2F_0|$.
    \item The fiber $G_0\in |2F_0|$ is simple and the only reducible fiber of $|2F_0|$.
    \item The fibration $|2F_0|$ is elliptic and not isotrivial with singular fibers of type~$\III^*,\I_1,\I_1$.
\end{enumerate}
\end{proposition}
\begin{proof}

For Claim (1), recall that by \Cref{lem: xxfixedlocus}, the support of the fiber $G_0\in |2F_0|$ is the union of the conductrix of $X$ and the divisorial part of the fixed locus of the numerically trivial involution of $X$, hence it is preserved by the whole $\Aut(X)$.

For Claim (2), assume by contradiction that $\sigma$ preserves the two half-fibers $F_0$, $F_0'$ of $|2F_0|$. Then, $\sigma$ fixes the base of $|2F_0|$, hence it induces an involution on the generic fiber $(F_0)_{\eta}$ of $|2F_0|$. By \cite{OguisoShioda}, $\MW(|2F_0|)$ is torsion-free, hence $\sigma$ is not a translation, so it has fixed points on $(F_0)_{\eta}$. This would produce an irreducible component of the fixed locus of $\sigma$ not contained in $G_0$, contradicting \Cref{lem: xxfixedlocus}.

For Claim (3), let $\sigma$ be the numerically trivial involution. By Claims (1) and (2) we have $\sigma(G_0)=G_0$ and $\sigma(F_0)=F_0'$, so $G_0$ is not a half-fiber. Moreover, if $|2F_0|$ has another reducible fiber $G_0'$, then $G_0$ and $G_0'$ are the only reducible fibers of $|2F_0|$ (otherwise the rank of the lattice spanned by their components would be too big). Then, $\sigma$ preserves $G_0$ and $G_0'$, hence it fixes the base of $|2F_0|$, contradicting the fact that $\sigma$ exchanges the two half-fibers of $|2F_0|$.
 
 For Claim (4), first note that $|2F_0|$ is not extremal by Claim (3), hence it is elliptic. By \cite{LangConfigs}, the singular fibers of~\(|2F_0|\) are either $\III^*,\I_1,\I_1$, or $G_0$ is the unique singular fiber of~\(|2F_0|\) and $|2F_0|$ is isotrivial with $j$-invariant $0$. Since $X$ is classical, \Cref{lem: halffibers} implies that $|2F_0|$ admits a smooth ordinary elliptic curve as half-fiber, so the latter case cannot occur.
\end{proof}

 We note the following immediate consequence of \Cref{prop: xximportantproperties} and \Cref{prop:Enriques.zero.entropy}.

\begin{theorem} \label{thm:XX}
Every Enriques surface of type~$\XX$ has zero entropy. More precisely, \(|2F_0|\) is the unique non-extremal genus \(1\) fibration on \(X\).
\end{theorem}

We are also able to describe the structure of the automorphism group of Enriques surfaces of type~$\XX$.

\begin{proposition} \label{thm:XX.automorphism.group}
If $X$ is an Enriques surface of type~$\XX$, then 
\[
    \Aut(X) \cong \IZ/2\IZ \times D_\infty.
\]
\end{proposition}
\begin{proof}
The subgroup $\Aut_{\IP^1}(X)$ of $\Aut(X)$ acting trivially on the base of $|2F_0|$ is isomorphic to $D_{\infty}$ by \Cref{lem: sign involution}, since $|2F_0|$ is elliptic and not isotrivial by \Cref{prop: xximportantproperties} and $\MW(|2F_0|) \cong \mathbb{Z}$ by \cite{OguisoShioda}.
By \Cref{prop: xximportantproperties} (2) and (4), the image of $\Aut(X) \to {\rm PGL}_2$ is generated by the numerically trivial involution $\sigma$. Since $\sigma$ is central in $\Aut(X)$, this yields the claim. 
\end{proof}

\section{Classification} \label{sec: Classification}

The aim of this section is to prove \Cref{thm: main}. Our strategy can be summarized as follows. By \Cref{prop:Enriques.zero.entropy}, there exists a unique genus $1$ fibration $|2F_0|$ with infinite Mordell--Weil group on $X$. Thus, $|2F_0|$ is preserved by all automorphisms of $X$ and hence, in particular, by the Mordell--Weil groups of the other genus $1$ fibrations on $X$. It turns out that this puts heavy restrictions on the Mordell--Weil groups that appear, eventually leading to the dual graphs of \Cref{thm: main}.

\begin{proposition} \label{prop: alternative}
Let $X$ be an Enriques surface with two genus~\(1\) fibrations $|2F_0|$ and $|2F_1|$ such that $F_0.F_1 = 1$. Suppose that $\MW(|2F_1|)$ preserves $|2F_0|$, and that $|2F_0|$ or $|2F_1|$ is elliptic. Then, $|2F_1|$ is extremal with reducible fibers
\[
    (\II^*), \; (\I_4^*), \; (\III^*,\I_2), \; (\III^*,\III), \; (\I_0^*,\I_0^*), \; (\I_2^*,\III,\III) \; \text{or} \; (\I_2^*,\I_2,\I_2).
\]
In particular, $\MW(|2F_1|) \cong (\mathbb{Z}/2\mathbb{Z})^a$ with $a \leq 2$.
\end{proposition}
\begin{proof}
Since $\MW(|2F_1|)$ preserves the numerical classes of $F_0$ and $F_1$, it acts on $\Num(X) \cong U \oplus E_8$ through the finite group ${\rm O}(E_8)$. Since the kernel of $\Aut(X) \to {\rm O}(\Num(X))$ is finite by \cite[Proposition~2.1]{DolgachevMartinNUM}, we conclude that $\MW(|2F_1|)$ is finite, that is, $|2F_1|$ is extremal.

Let $G_1 \in |2F_1|$ be a reducible fiber. Since $F_0.F_1 = 1$, there exist either at most two simple components of $G_1$ or one double component of $G_1$ that meets $F_0$. Since $\MW(|2F_1|)$ preserves the numerical classes of $F_0$ and $F_1$, the set of such components is preserved by $\MW(|2F_1|)$. 


If $G_1$ is simple or multiplicative, we understand the action of $\MW(|2F_1|)$ on~$G_1$ by \Cref{lem: MWonJAC,lem: action.on.In}, and \Cref{tab:elliptic,tab:quasi-elliptic}. 
Combining this with the previous paragraph, we see that either a simple component of $G_1$ has a $\MW(|2F_1|)$-orbit of length $\le 2$, or a double component of $G_1$ has a trivial $\MW(|2F_1|)$-orbit. 
In particular, $G_1$ must be of type~$\II^*,\III^*,\I_4^*,\I_2^*,\I_0^*,\III,$ or $\I_2$.
In order to see this, assume for instance that $G_1$ is simple of type~$\I_n^*$. The orbit of a simple component of $G_1$ has length $\le 2$ if and only if $n=4$, and a double component has a trivial orbit if and only if $n$ is even (the fixed component is the central one). The other cases are analogous.

By \Cref{tab:elliptic}, the previous discussion already gives the desired claim if $|2F_1|$ is elliptic. Thus assume that $|2F_1|$ is quasi-elliptic; in particular, $p=2$ and the half-fiber $F_1$ is additive.
There is a subgroup $H \subseteq \MW(|2F_1|)$ of index at most $2$ (resp. at most $1$ if $X$ is not classical) and which preserves the half-fiber~$F_0$. Let $E_1$ be the component of $F_1$ meeting $F_0$. Since $H$ fixes $F_0 \cap E_1$ and any singular point of $F_1$ on $E_1$ and these points are distinct by \cite[Lemma 3.5]{DolgachevMartinNUM}, it fixes $E_1$ pointwise. Here, we use again the fact that involutions of $\mathbb{P}^1$ and the cuspidal cubic have only one fixed point in characteristic $2$. Thus, $H$ fixes the base of $|2F_0|$ and acts with a fixed point on a general fiber of $|2F_0|$, hence, as $|2F_0|$ is elliptic, $H$ contains at most one non-trivial involution (cf. \Cref{lem: sign involution}).
Thus, $\MW(|2F_1|)\cong (\IZ/2\IZ)^a$ with $a\le 2$ (resp. $a \leq 1$ if $X$ is not classical) and the claim follows by \Cref{tab:quasi-elliptic}.
\end{proof}

\begin{corollary} \label{cor: alternative}
In the setting of \Cref{prop: alternative}, assume that $|2F_0|$ is not extremal. Then, $|2F_1|$ is extremal with reducible fibers
\[
    (\I_4^*), \; (\III^*,\III), \; (\III^*,\I_2), \; (\I_2^*,\III,\III) \; \text{or} \; (\I_2^*,\I_2,\I_2).
\]
Moreover, $|2F_1|$ admits a simple reducible fiber $G_1$ such that $F_0$ meets two distinct simple components of $G_1$. 
\end{corollary}
\begin{proof}
It suffices to prove the last statement. Indeed, a fiber of type~$\II^*$ has only one simple component and we know from the proof of \Cref{prop: alternative} that a simple fiber of type~$\I_0^*$ can only appear if its central component meets $F_0$.

To prove the last statement, it suffices to note that if $F_0$ meets only one component of every fiber of $|2F_1|$, then the lattice spanned by fiber components of $|2F_1|$ that are orthogonal to $F_0$ has rank $8$, hence $|2F_0|$ is extremal. Thus, there must be a fiber of $|2F_1|$ which has two distinct components meeting $F_0$, and this fiber is necessarily simple and reducible. 
\end{proof}

Recall that by \Cref{prop:Enriques.zero.entropy}, any Enriques surface of zero entropy with infinite automorphism group admits a unique non-extremal genus~\(1\) fibration (necessarily elliptic), which we always denote by \(|2F_0|\). Being preserved by the whole \(\Aut(X)\), the fibration \(|2F_0|\) is preserved by \(\MW(|2F|)\) for every fibration \(|2F|\) on~\(X\).

\begin{lemma} \label{lem: excludeXX}
Let $X$ be an Enriques surface of zero entropy with infinite automorphism group. Let $|2F_0|$ be the unique non-extremal fibration and let $F_1$ be a half-fiber with $F_0.F_1 = 1$. Assume that $X$ is not of type~$\XX$. Then, $|2F_1|$ is extremal with reducible fibers
\[
    (\I_4^*), \; (\III^*,\III) \; \text{or} \; (\III^*,\I_2).
\]
\end{lemma}
\begin{proof}
By \Cref{cor: alternative}, we have to show that if the reducible fibers of $|2F_1|$ are of type~$(\I_2^*,\III,\III)$ or $(\I_2^*,\I_2,\I_2)$, then $X$ is of type~$\XX$. Denote by $G_1\in |2F_1|$ the fiber of type~$\I_2^*$.

Assume first that the fiber $G_1$ is simple.
By the proof of \Cref{prop: alternative}, the union $\Gamma$ of the components of \(G_1\) orthogonal to $F_0$ is of type~$A_3\cup A_3$. 
If $\Gamma$ is contained in a single fiber $G_0 \in |2F_0|$, then $G_1$ must be simple of type~$\I_8$, since the central component of $G_1$ is a bisection of $|2F_0|$, so the adjacent components must be simple in $G_0$. 
If instead there are two fibers $G_0,G_0'\in |2F_0|$ containing $\Gamma$, then, one of them, say $G_0$, must be of type~$\I_4$, for otherwise $|2F_0|$ would be extremal. We get three possible diagrams (the last two according to whether \(G_0\) is double or simple):

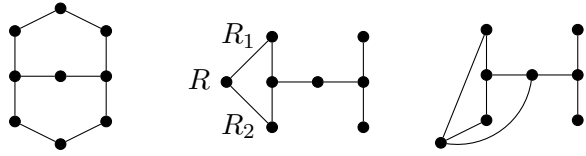
\begin{center}
\begin{tikzpicture}[scale=0.6]
    \node (R1) at (0,0) [nodal] {};
    \node (R3) at (2,0) [nodal] {};
    \node (R4) at (2,-1) [nodal] {};
    \node (R5) at (2,-2) [nodal] {};
    \node (R7) at (0,-2) [nodal] {};
    \node (R8) at (0,-1) [nodal] {};
    \node (R9) at (1,-1) [nodal] {};
    \node (RX) at (1,0.5) [nodal] {};
    \node (RXX) at (1,-2.5) [nodal] {};
    \draw (R9)--(R8) (R9)--(R4) (R3)--(R5) (R7)--(R1) (R1)--(RX)--(R3) (R7)--(RXX)--(R5);
\end{tikzpicture}
\qquad
\begin{tikzpicture}[scale=0.6]
    \node (R1) at (0,0) [nodal, label=left:$R_1$] {};
    \node (R3) at (2,0) [nodal] {};
    \node (R4) at (2,-1) [nodal] {};
    \node (R5) at (2,-2) [nodal] {};
    \node (R7) at (0,-2) [nodal, label=left:$R_2$] {};
    \node (R8) at (0,-1) [nodal] {};
    \node (R9) at (1,-1) [nodal] {};
    \node (RXX) at (-1,-1) [nodal, label=left:$R$] {};
    \draw (R9)--(R8) (R9)--(R4) (R3)--(R5) (R7)--(R1) (R1)--(RXX)--(R7);
\end{tikzpicture}
\qquad
\begin{tikzpicture}[scale=0.6]
    \node (R1) at (0,0) [nodal] {};
    \node (R3) at (2,0) [nodal] {};
    \node (R4) at (2,-1) [nodal] {};
    \node (R5) at (2,-2) [nodal] {};
    \node (R7) at (0,-2) [nodal] {};
    \node (R8) at (0,-1) [nodal] {};
    \node (R9) at (1,-1) [nodal] {};
    \node (RXX) at (-1,-2.5) [nodal] {};
    \draw (R9)--(R8) (R9)--(R4) (R3)--(R5) (R7)--(R1) (R1)--(RXX)--(R7);
    \draw (RXX) to[bend right=45] (R9);
\end{tikzpicture}
\end{center}
In the first (resp. third) diagram, we find a half-fiber $F_2$ of type~$\I_6$ (resp. $\I_4$) such that $F_0.F_2=1$, contradicting \Cref{cor: alternative}.
In the second diagram, the $(-2)$-curve~$R$ is a bisection of $|2F_1|$. Let $G_1'$ and $G_1''$ be the other two reducible fibers in $|2F_1|$, with components $R_3,R_3'$ and $R_4,R_4'$ respectively. If $G_1'$ (or $G_1''$) is double, then $R$ meets one of its components with multiplicity $1$, say $R_3$ (resp. $R_4$). If instead $G_1'$ is simple, then, by \Cref{cor: actiononsimplefiber}, there is an element in $\MW(|2F_1|)$ exchanging its components, and therefore $F_0.R_3=F_0.R_3'$. In particular, $R.R_3=R.R_3'=1$. In both cases, we have that $R.R_3=R.R_4=1$, obtaining the following diagram:
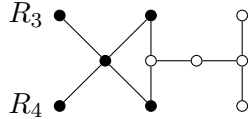
\[
\begin{tikzpicture}[scale=0.6]
    \node (A) at (-2,0) [nodal,label=left:$R_3$] {};
    \node (B) at (-2,-2) [nodal,label=left:$R_4$] {};
    \node (R1) at (0,0) [nodal] {};
    \node (R3) at (2,0) [nodal,fill=white] {};
    \node (R4) at (2,-1) [nodal,fill=white] {};
    \node (R5) at (2,-2) [nodal,fill=white] {};
    \node (R7) at (0,-2) [nodal] {};
    \node (R8) at (0,-1) [nodal,fill=white] {};
    \node (R9) at (1,-1) [nodal,fill=white] {};;
    \node (RXX) at (-1,-1) [nodal] {};
    \draw (R9)--(R8) (R9)--(R4) (R3)--(R4)--(R5) (R7)--(R8)--(R1) (R1)--(RXX)--(R7) (A)--(RXX)--(B);
\end{tikzpicture}
\]
However, the simple fiber $G_2$ in bold of type~$\I_0^*$ satisfies $G_2.F_0=2$, contradicting \Cref{cor: alternative}.

On the other hand, assume that $G_1\in |2F_1|$ is a double fiber, say \(G_1 = 2F_1\).
Then $p=2$ and the fibration $|2F_1|$ is quasi-elliptic by \cite[§2A]{Lang2} or alternatively by \Cref{tab:elliptic}, because the $2$-torsion subgroup of an elliptic curve in characteristic~$2$ has order at most~$2$. Denote by $R$ its curve of cusps.
By the last paragraph of the proof of \Cref{prop: alternative}, $|2F_0|$ has two half-fibers, hence $X$ is classical. Thus, by \Cref{lem: xxgeneralities}, it suffices to prove that the second half-fiber $F_1'$ of $|2F_1|$ is reducible (and therefore of type~$\III$).
If by contradiction $F_1'$ is irreducible, then every $g\in \MW(|2F_1|)$ preserves its singular point and the smooth point $F_1'\cap R$, so $\MW(|2F_1|)$ acts trivially on $F_1'$. Since $F_1'$ is a bisection of $|2F_0|$, the group $\MW(|2F_1|)$ fixes the base of the elliptic fibration $|2F_0|$ pointwise, and it fixes the point(s) $G_0\cap F_1$ for a general $G_0\in |2F_0|$, so $|{\MW(|2F_1|)}| \le 2$ by \autoref{lem: sign involution}, a contradiction. 
\end{proof}

By \Cref{thm:AA,thm:EE,thm:XX}, every Enriques surface of type \(\AA\), \(\EE\) or \(\XX\) has zero entropy.
Thus, to finish the proof of \Cref{thm: main}, it suffices by \Cref{prop:Enriques.zero.entropy} to show that if an Enriques surface $X$ admits a unique genus $1$ fibration with infinite Mordell--Weil group, then $X$ is of type \(\AA\), \(\EE\) or \(\XX\).

\begin{proof}[Proof of \Cref{thm: main}]
Let \(X\) be an Enriques surface  with a unique non-extremal fibration \(|2F_0|\).
In particular, \(|2F_0|\) is preserved by every automorphism of \(X\). By \cite[Theorem~3.2.1]{Cossec.Dolgachev} or \cite[Theorem~2.3.3]{Enriques_I}, there exists a half-fiber \(F_1\) with \(F_0.F_1 = 1\).

By \Cref{lem: excludeXX}, we may assume that there is a reducible fiber $G_1\in |2F_1|$ of type~$\I_4^*$ or $\III^*$.
Let $\Gamma$ be the union of the components of \(G_1\) orthogonal to $F_0$, and let $G_0 \in |2F_0|$ be the fiber containing $\Gamma$.

Assume first that $G_1$ is of type~$\I_4^*$. By \Cref{cor: alternative} and the proof of \Cref{prop: alternative}, $G_1$ is simple and $\Gamma$ is of type~$A_7$. Since \(|2F_0|\) is not extremal, $G_0$ is a (simple or double) fiber of type~$\I_8$, or a simple fiber of type~$\III^*$, leading to the following three possible graphs.

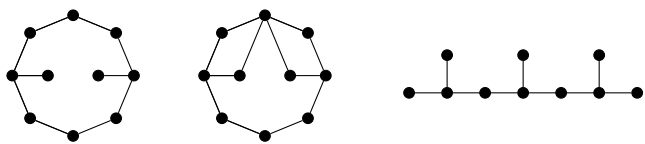
\[
\begin{tikzpicture}[scale=0.4]
    \node (R1) at (180:2) [nodal] {};
    \node (R2) at (135:2) [nodal] {};
    \node (R3) at (90:2) [nodal] {};
    \node (R4) at (45:2) [nodal] {};
    \node (R5) at (0:2) [nodal] {};
    \node (R6) at (315:2) [nodal] {};
    \node (R7) at (270:2) [nodal] {};
    \node (R8) at (225:2) [nodal] {};
    \node (R9) at (intersection of R2--R7 and R3--R8) [nodal] {};
    \node (R10) at (intersection of R4--R7 and R3--R6) [nodal] {};
    \draw (R6)--(R7)--(R8)--(R1)--(R2)--(R3)--(R4) (R1)--(R9) (R5)--(R10) (R4)--(R5)--(R6);
    \draw (R7)--(R8)--(R1)--(R2)--(R3)--(R4) (R1)--(R9) ;
\end{tikzpicture}
\qquad
\begin{tikzpicture}[scale=0.4]
    \node (R1) at (180:2) [nodal] {};
    \node (R2) at (135:2) [nodal] {};
    \node (R3) at (90:2) [nodal] {};
    \node (R4) at (45:2) [nodal] {};
    \node (R5) at (0:2) [nodal] {};
    \node (R6) at (315:2) [nodal] {};
    \node (R7) at (270:2) [nodal] {};
    \node (R8) at (225:2) [nodal] {};
    \node (R9) at (intersection of R2--R7 and R3--R8) [nodal] {};
    \node (R10) at (intersection of R4--R7 and R3--R6) [nodal] {};
    \draw (R6)--(R7)--(R8)--(R1)--(R2)--(R3)--(R4) (R1)--(R9) (R5)--(R10) (R4)--(R5)--(R6);
    \draw (R7)--(R8)--(R1)--(R2)--(R3)--(R4) (R1)--(R9) (R3)--(R9) (R3)--(R10);
\end{tikzpicture}
\qquad
\begin{tikzpicture}[scale=0.5]
    \clip (-2.5,-1.3) rectangle (4.5,1.5);
    \node (R4) at (0,0) [nodal] {};
    \node (R5) at (1,1) [nodal] {};
    \node (R6) at (1,0) [nodal] {};
    \node (R7) at (2,0) [nodal] {};
    \node (R8) at (3,0) [nodal] {};
    \node (R9) at (4,0) [nodal] {};
    \node (R3) at (-1,0) [nodal] {};
    \node (R2) at (-2,0) [nodal] {};
    \node (R1) at (3,1) [nodal] {};
    \node (RX) at (-1,1) [nodal] {};
\draw (R2)--(R3)--(R6) (R5)--(R6)--(R9) (R1)--(R8) (R3)--(RX);
\end{tikzpicture}
\]
In the first case, $X$ is of type \(\AA\). In the second case, any half-fiber $F_2$ of type~$\I_4$ in the graph satisfies $F_0.F_2=1$, contradicting \Cref{prop: alternative}. In the third case, any half-fiber $F_2$ of type~$\I_2^*$ in the graph satisfies $F_0.F_2=1$, so $X$ is of type~$\XX$ by \Cref{lem: excludeXX}.

Assume now that $G_1$ is of type~$\III^*$. If $G_1$ is a double fiber, then $\Gamma$ is of type~$E_7$. Since $|2F_0|$ is not extremal, we deduce that $G_0$ is of type~$\III^*$. Since $G_0$ and~$G_1$ share components, $G_0$ must be simple by \cite[Lemma~3.5]{DolgachevMartinNUM}. Thus, $X$ contains \((-2)\)-curves with the following dual graph:
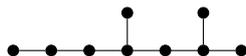
\begin{center}
\begin{tikzpicture}[scale=0.5]
    \node (R4) at (0,0) [nodal] {};
    \node (R5) at (1,1) [nodal] {};
    \node (R6) at (1,0) [nodal] {};
    \node (R7) at (2,0) [nodal] {};
    \node (R8) at (3,0) [nodal] {};
    \node (R9) at (4,0) [nodal] {};
    \node (R3) at (-1,0) [nodal] {};
    \node (R2) at (-2,0) [nodal] {};
    \node (R1) at (3,1) [nodal] {};
\draw (R2)--(R3)--(R6) (R5)--(R6)--(R9) (R1)--(R8);
\end{tikzpicture}
\end{center}
As in the previous case, the half-fiber $F_2$ of type~$\I_2^*$ satisfies $F_0.F_2=1$, so $X$ is of type~$\XX$ by \Cref{lem: excludeXX}.

If instead $G_1$ is a simple fiber, then $\Gamma$ is of type~$A_7$ or $E_6$, by \Cref{cor: actiononsimplefiber}, \Cref{tab:elliptic} and \Cref{tab:quasi-elliptic}. In the first case, $G_0$ is a (double or simple) fiber of type~$\I_8$, while in the second case, $G_0$ is a (double or simple) fiber of type~$\IV^*$, or a simple fiber of type~$\III^*$. We get the following possible dual graphs:
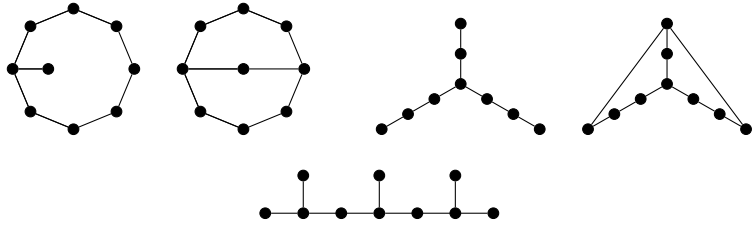
\begin{center}
\begin{tikzpicture}[scale=0.4]
    \node (R1) at (180:2) [nodal] {};
    \node (R2) at (135:2) [nodal] {};
    \node (R3) at (90:2) [nodal] {};
    \node (R4) at (45:2) [nodal] {};
    \node (R5) at (0:2) [nodal] {};
    \node (R6) at (315:2) [nodal] {};
    \node (R7) at (270:2) [nodal] {};
    \node (R8) at (225:2) [nodal] {};
    \node (R9) at (intersection of R2--R7 and R3--R8) [nodal] {};
    \draw (R6)--(R7)--(R8)--(R1)--(R2)--(R3)--(R4) (R1)--(R9) (R4)--(R5)--(R6);
    \draw (R7)--(R8)--(R1)--(R2)--(R3)--(R4) (R1)--(R9) ;
\end{tikzpicture}
\quad
\begin{tikzpicture}[scale=0.4]
    \node (R1) at (180:2) [nodal] {};
    \node (R2) at (135:2) [nodal] {};
    \node (R3) at (90:2) [nodal] {};
    \node (R4) at (45:2) [nodal] {};
    \node (R5) at (0:2) [nodal] {};
    \node (R6) at (315:2) [nodal] {};
    \node (R7) at (270:2) [nodal] {};
    \node (R8) at (225:2) [nodal] {};
    \node (R9) at (0,0)[nodal] {};
    \draw (R6)--(R7)--(R8)--(R1)--(R2)--(R3)--(R4) (R1)--(R9)--(R5) (R4)--(R5)--(R6);
    \draw (R7)--(R8)--(R1)--(R2)--(R3)--(R4) (R1)--(R9) ;
\end{tikzpicture}
\qquad
    \begin{tikzpicture}[scale=0.8]
        \node (R1) at (90:1) [nodal] {};
        \node (R2) at (90:0.5) [nodal] {};
        \node (R3) at (0:0) [nodal] {};
        \node (R4) at (210:0.5) [nodal] {};
        \node (R5) at (210:1) [nodal] {};
        \node (R6) at (210:1.5) [nodal] {};
        \node (R7) at (330:0.5) [nodal] {};
        \node (R8) at (330:1) [nodal] {};
        \node (R9) at (330:1.5) [nodal] {};
        \draw (R1)--(R2)--(R3)--(R4)--(R5)--(R6) (R3)--(R7)--(R8)--(R9);
    \end{tikzpicture}
\quad
    \begin{tikzpicture}[scale=0.8]
        \node (R1) at (90:1) [nodal] {};
        \node (R2) at (90:0.5) [nodal] {};
        \node (R3) at (0:0) [nodal] {};
        \node (R4) at (210:0.5) [nodal] {};
        \node (R5) at (210:1) [nodal] {};
        \node (R6) at (210:1.5) [nodal] {};
        \node (R7) at (330:0.5) [nodal] {};
        \node (R8) at (330:1) [nodal] {};
        \node (R9) at (330:1.5) [nodal] {};
        \draw (R1)--(R2)--(R3)--(R4)--(R5)--(R6) (R3)--(R7)--(R8)--(R9) (R1)--(R9) (R1)--(R6);
    \end{tikzpicture}
\qquad
\begin{tikzpicture}[scale=0.5]
    \clip(-2.5,-0.2) rectangle (4.5,2);
    \node (R4) at (0,0) [nodal] {};
    \node (R5) at (1,1) [nodal] {};
    \node (R6) at (1,0) [nodal] {};
    \node (R7) at (2,0) [nodal] {};
    \node (R8) at (3,0) [nodal] {};
    \node (R9) at (4,0) [nodal] {};
    \node (R3) at (-1,0) [nodal] {};
    \node (R2) at (-2,0) [nodal] {};
    \node (R1) at (3,1) [nodal] {};
    \node (RX) at (-1,1) [nodal] {};
\draw (R2)--(R3)--(R6) (R5)--(R6)--(R9) (R1)--(R8) (R3)--(RX);
\end{tikzpicture}
\end{center}
In the second and fourth graph, we find a half-fiber \(F_2\) of type \(\I_6\) with \(F_0.F_2 = 1\), contradicting again \Cref{prop: alternative}.
As above, the last graph leads to Enriques surfaces of type \(\XX\) by \Cref{lem: excludeXX}.
In the first and third case, the component of $G_0$ that is not contained in $G_1$ meets at least one component of the other reducible fiber of $|2F_1|$ transversally by \Cref{cor: alternative}. Thus, in these cases, we obtain the graph of types \(\AA\) and \(\EE\)), respectively.
\end{proof}

\renewcommand{\discretionary}{}
\printbibliography

\end{document}